\documentclass[10pt,a4paper]{amsart}

\usepackage[english]{babel}

\usepackage[utf8]{inputenc}
\usepackage[T1]{fontenc}
\usepackage{lmodern}

\usepackage{amsthm}
\usepackage{amsmath}
\usepackage{amsfonts}
\usepackage{amssymb}

\usepackage{xcolor}
\usepackage{color}
\usepackage{hyperref}

\usepackage[headings]{fullpage}

\newtheorem{theo}{Theorem}[subsection]
\newtheorem*{theo*}{Theorem}
\newtheorem{prop}[theo]{Proposition}
\newtheorem{lemm}[theo]{Lemma}
\newtheorem{coro}[theo]{Corollary}

\newtheorem{defi}[theo]{Definition}

\theoremstyle{remark}
\newtheorem{rema}[theo]{Remark}
\newtheorem{fact}[theo]{Fact}

\newcommand{\Z}{\mathbb{Z}}
\newcommand{\Q}{\mathbb{Q}}
\newcommand{\F}{\mathbb{F}}
\newcommand{\eps}{\varepsilon}

\newcommand{\M}{\mathcal{M}}
\newcommand{\T}{\mathcal{T}}
\newcommand{\V}{\mathcal{V}}
\newcommand{\I}{\mathcal{I}}
\newcommand{\B}{\mathcal{B}}

\newcommand{\s}{\mathfrak{S}}

\renewcommand{\bar}[1]{\overline{#1}}

\newcommand{\Hom}{\operatorname{Hom}}
\newcommand{\im}{\operatorname{im}}
\newcommand{\GL}{\operatorname{GL}}
\newcommand{\SL}{\operatorname{SL}}
\newcommand{\St}{\operatorname{St}}
\newcommand{\unr}{\operatorname{unr}}
\newcommand{\ind}{\operatorname{ind}}
\newcommand{\soc}{\operatorname{soc}}

\newcommand{\Sym}{\operatorname{Symm}}
\newcommand{\JH}{\operatorname{JH}}
\newcommand{\jh}{\mathcal{J}\mathcal{H}}

\newcommand{\K}{\mathcal{K}}

\newcommand{\matr}[4]{\begin{pmatrix}#1 & #2 \\ #3 & #4\end{pmatrix}}
\newcommand{\smatr}[4]{\left(\begin{smallmatrix}#1 & #2 \\ #3 & #4\end{smallmatrix}\right)}

\title[An algorithm]{An algorithm for computing the reduction of $2$-dimensional crystalline
representations of $\text{Gal}(\bar\Q_p/\Q_p)$}
\author{Sandra Rozensztajn}
\address{UMPA, \'ENS de Lyon\\
UMR 5669 du CNRS\\
46, all\'ee d'Italie\\
69364 Lyon Cedex 07\\
France}
\email{sandra.rozensztajn@ens-lyon.fr}

\begin{document}

\begin{abstract}
We describe an algorithm to compute the reduction modulo $p$ of a
crystalline Galois representation of dimension $2$ of
$\text{Gal}(\bar\Q_p/\Q_p)$ with distinct Hodge-Tate weights via the
semi-simple modulo $p$ Langlands correspondence. We give
some examples computed with an implementation of this algorithm in SAGE.
\end{abstract}
  
\maketitle
\tableofcontents

\section{Introduction and notation}

Let $V$ be an irreducible crystalline representation of 
$\text{Gal}(\bar\Q_p/\Q_p)$, with distinct Hodge-Tate weights. 
Our goal is to give an algorithm to compute the semi-simplification
$\bar{V}^{ss}$ of the reduction modulo $p$ of $V$.
There are several ways to approach this problem. The first is to use
global methods, using congruences between modular forms. This method has
already been used extensively by Savitt-Stein and Buzzard, see for
example \cite[Paragraph 6.2]{Br03b}.  Another way is to use local methods coming from
$p$-adic Hodge theory. One such algorithm is described in \cite{Ber12}
and uses $(\phi,\Gamma)$ modules, another one is described in \cite{CL13}
and makes use of Breuil-Kisin modules. Both these methods have the
drawback that they require to work with complicated objects, and they
have not yet been implemented.

The algorithm we describe requires only to do linear algebra in simple
rings such as $\Z/p^n\Z$.
For this we use the modulo $p$ Langlands correspondence, and its
compatibility with the $p$-adic Langlands correspondence. 
The idea of using this correspondence as a tool to compute the reduction
modulo $p$ of Galois representations appears first in \cite{BG09}, and
has since been used several times (see \cite{BG13,GG,BG15,BGR,Ars}) to do
computations in small slope.  The correspondence can in fact be used to
compute the reduction in any slope. We describe the general method, and
how it can be implemented on a computer.

The crystalline representation $V$ we are interested in can be described (up to
twist by a crystalline character) by an integer $k \geq 2$, which is such that $V$ has Hodge-Tate
weights $0$ and $k-1$, and an element $a_p$ of $\bar\Q_p$ with positive
valuation. We denote by $V_{k,a_p}$ this irreducible crystalline
representation of $\text{Gal}(\bar\Q_p/\Q_p)$.
By the modulo $p$ Langlands correspondence, the reduction
$\bar{V}_{k,a_p}^{ss}$ we are looking for is determined by a smooth
representation $\bar\Theta_{k,a_p}$ of $\GL_2(\Q_p)$ over
$\bar\F_p$, and we can give an explicit formula for
$\bar\Theta_{k,a_p}$ in terms of $k$ and $a_p$.  So the problem is
to understand $\bar\Theta_{k,a_p}$ as a
representation of $\GL_2(\Q_p)$ from this formula. We explain how this
problem can be reduced to linear algebra, and more precisely to a finite 
number of questions about some explicit vectors belonging to some
explicit subspace of a fixed vector space. This is the object of Sections
\ref{computefromrel} and \ref{filtration}, after a few reminders on the
$p$-adic and modulo $p$ Langlands correspondences in Sections
\ref{irreduciblerepr} and \ref{correspondences}.

One difficulty is that the representation 
$\bar\Theta_{k,a_p}$ is of infinite dimension, so we have to
understand how to compute it from finite-dimensional data. This gives
rise to the main limitation of this algorithm, which is that only very
small values of the prime $p$, and not too large values of the weight $k$,
allow for manageable dimensions for the vector spaces we manipulate.  
This still allows however for the computation of non-trivial examples.
We give some examples computed thanks to an implementation of the
algorithm using SAGE (\cite{sage}) in Section \ref{examples}.

We show moreover that the result of the computations of the algorithm
give additional information about the reduction. Indeed, we know that
$\bar{V}^{ss}_{k,a_p}$ is locally constant with respect to $a_p$
(\cite{Ber12} and \cite{BLZ}), and with respect to the weight $k$ when $a_p
\neq 0$ (\cite{Ber12}). The output of the computation gives us a radius
of constancy with respect to $a_p$, and, under an additional condition that
is often satisfied in the examples, with respect to the weight when $a_p
\neq 0$. 

\begin{theo*}
Suppose that the algorithm described has allowed us to compute 
$\bar{V}^{ss}_{k,a_p}$ by using denominators of valuation at most
$\delta$. Then:
\begin{enumerate}
\item
For all $a_p'$ such that $v_p(a_p-a_p') > \delta$, 
$\bar{V}^{ss}_{k,a_p}$ and $\bar{V}^{ss}_{k,a_p'}$ are isomorphic.

\item
Let $c$ be the smallest integer strictly larger than 
$v_p(a_p) + \delta$. If $c \leq
(k-2)/(p+1)$, then for all $k' > k$ such that $k' \equiv k \mod
(p-1)p^{1+\lfloor \delta \rfloor + \lfloor \log_p(c) \rfloor}$, 
$\bar{V}^{ss}_{k,a_p}$ and $\bar{V}^{ss}_{k',a_p}$ are isomorphic.
\end{enumerate}
\end{theo*}

This result is the object of Section \ref{locconst}, the first part being
much easier than the second one.
In the case of the local constancy with respect to $a_p$, bounds were
already known thanks to the results of
\cite{Ber12} and \cite{BLZ}, and our examples show that
these bounds are generally not optimal. In the case of local constancy
with respect to the weight, the result of \cite{Ber12} was not effective
at all, so this is the first time we have some effective results in this
direction except for the values of $a_p$ for which $\bar{V}^{ss}_{k,a_p}$
has been computed for all weights $k$.

\subsection{Notation}

Denote by $G$ the group $\GL_2(\Q_p)$, $Z$ its center, and $K$ the subgroup
$\GL_2(\Z_p)$.

We write $G_{\Q_p}$ for the group $\text{Gal}(\bar\Q_p/\Q_p)$.
Let $a_p$ be in $\bar\Q_p$ with positive valuation, and $k \geq 2$ be an
integer. We denote by $V_{k,a_p}$ be the irreducible crystalline representation of
$G_{\Q_p}$ of dimension $2$ with Hodge-Tate weights $\{0,k-1\}$ and such
that the trace of $\phi$ on the associated filtered $\phi$-module is
$a_p$.

Let $\omega$ be the mod $p$ cyclotomic character, or the character of
$\Q_p^\times$ corresponding to it via local class field theory. For any
$\lambda \in \bar\F_p^\times$, let $\unr(\lambda)$ be the unramified character
of $\Q_p^\times$ such that $\unr(\lambda)(p) = \lambda$.
We also denote by $\omega_2$ a Serre fundamental character of level $2$ on
the inertia subgroup.

Let $v$ be the valuation
on $\bar\Q_p$ normalized by $v(p) = 1$.

\medskip

Let $\K$ be a finite extension of $\Q_p$, with ring of integers $R$ and
residue field $E$. When we have fixed a value of $a_p$, we will assume
that $\K$ contains $a_p$.

\medskip

Let $A$ by any $\Z_p$-algebra.
For any non-negative integer $s$, let $\Sym^s A^2$ be the
representation of $K$ coming from the standard action of $K$ on $A^2$.
For any representation $V$ of $\GL_2(\Z_p)$, let $V(i) = V\otimes\det^i$.
For $a \geq 0$ and $b\in \Z/(p-1)\Z$, let 
$\sigma_a(b) = \Sym^a\F_p^2 \otimes\det^b$. 
Then the $\sigma_a(b)$ for $0 \leq a \leq p-1$ are exactly the
irreducible representations of $K$ with coefficients in $\F_p$,
and any irreducible representation of $K$ on a field of characteristic
$p$ comes from these representations by base change.
We will also denote by $\sigma_a$ and $\sigma_a(b)$ the base change of this
representation to another field of characteristic $p$, or 
$\sigma_{a,E}$ and $\sigma_a(b)_E$
if it is necessary to indicate the field $E$.

\medskip

If $n$ is an integer, $[n]$ denotes the unique integer in
$\{1,\dots,p-1\}$ that is congruent to $n$ modulo $p-1$.

\section{Irreducible representations of $G$ in characteristic $p$}
\label{irreduciblerepr}

\subsection{The tree}

Let $\T$ be the tree of $\SL_2(\Q_p)$. Its vertices are indexed by the
homothety classes of lattices in $\Q_p^2$, and there is an edge between
$[L]$ and $[L']$ if and only if there are representatives with $L' \subset L$
and $[L:L'] = p$. The vertices are also indexed by $G/KZ$, and $G$ acts
on the tree by left translation.

We give a system of representatives of $G/KZ$, following \cite[Paragraph
2.2]{Br03a} for $F = \Q_p$. 
Let $\mu_{p-1}$ the set of
$(p-1)$-th roots of unity in $\Z_p$.
Let $I_0 = \{0\}$, $I_1 = \{0\} \cup \mu_{p-1}$, and for any $n \geq 1$, 
$I_n = \{\sum_{i=0}^{n-1}\lambda_ip^i, \lambda_i \in I_1\}$. Then $I_n$ has
cardinality $p^n$.

If $\mu \in I_n$, we let $g^0_{n,\mu} = \matr{p^n}{\mu}01$, and
$g^1_{n,\mu} = \matr 10{p\mu}{p^{n+1}}$. Then the set of the $g^{\eps}_{n,\mu}$ for
$\eps\in \{0,1\}$, $n \geq 0$, $\mu \in I_n$ form a system of
representatives of $G/KZ$ and so is naturally in bijection with the set
of vertices of $\T$.

Moreover, $g^0_{n,\mu}$ is at distance $n$ from $1 = g^0_{0,0}$, and
$g^1_{n,\mu}$ is at distance $n+1$ from $1$. 
We denote by $C_n$ the circle of radius $n$ with center in $1$, and $B_n$
the ball of radius $n$ with center in $1$.

\subsection{Compact induction}
\label{compactind}

Let $(\rho,V)$ be a finite-dimensional representation of $KZ$ with
coefficients in some ring. We denote
by $I(V)$ the compact induction $\ind_{KZ}^GV$. We denote by $[g,v]$,
$g\in G$ and $v \in V$ elements of this module, where
$[g,v]$ is the function that sends $\gamma$ to $\rho(\gamma g)v$ if
$\gamma g \in KZ$, and sends $\gamma$ to $0$ if $\gamma g \not\in KZ$.
The action of $G$ on these elements is given by
$\gamma[g,v] = [\gamma g,v]$, and moreover for all $\kappa \in KZ$ we
have that $[g\kappa,v] = [g,\rho(\kappa)v]$.

Let $X$ be a subset of the set of vertices of $\T$. We say that $f \in
I(V)$ has support in $X$ if it can be written as a (finite) sum
$\sum [g_i,v_i]$ with the class of $g_i$ in $X$.

Let $V$ be a representation of $K$ of the form $\sigma_r(s)$. 
We let $\matr p00p$ act trivially on $V$, so that it becomes a
representation of $KZ$.
There is a $G$-equivariant Hecke operator $T$ acting on $I(V)$, defined in
\cite{BL94} (see Paragraph \ref{descrT} for an explicit description of
this operator).

\subsection{Irreducible representations of $G$ over $\bar\F_p$}

In this Paragraph all representations have coefficients in $\bar\F_p$.

If $r$ is an integer in $\{0,\dots,p-1\}$, $\chi$ a smooth character of
$\Q_p^\times$, and $\lambda\in \bar\F_p^\times$, we denote by
$\pi(r,\lambda,\chi)$ the representation
$\left(I(\sigma_r)/(T-\lambda)\right)\otimes\left(\chi\circ\det\right)$.
Note that $I(\sigma_a(b))/(T-\lambda) = \pi(a,\lambda,\omega^b)$.

We can make the list of  irreducible smooth admissible representations of $G$ with
coefficients in $\bar\F_p$, following \cite{BL94}, \cite{BL95} and \cite{Br03a}.
The irreducible representations are:
\begin{enumerate}
\item
$\pi(r,0,\chi)$

\item
$\pi(r,\lambda,\chi)$ for $\lambda\in \bar\F_p^\times$ and $(r,\lambda)\not\in
\{(0,\pm 1),(p-1,\pm 1)\}$.

\item
$\St \otimes \left(\chi\circ\det\right)$

\item
$\chi\circ\det$
\end{enumerate}
Moreover, 
$\pi(r,0,\chi) = \pi(r,0,\chi\unr(-1)) = \pi(p-1-r,0,\chi\omega^r)$ ;
$\pi(r,\lambda,\chi) = \pi(r,-\lambda,\chi\unr(-1))$ for any $r$ and
$\lambda$ ;  
$\pi(0,\lambda,\chi) = \pi(p-1,\lambda,\chi)$ 
for any $\lambda\neq \pm 1$ ;
and 
$\pi(0,\lambda,\chi)^{ss} = \pi(p-1,\lambda,\chi)^{ss}
= \St\otimes\left(\chi\unr(\lambda)\circ\det\right)
\oplus\left(\chi\unr(\lambda)\circ\det\right)$
if $\lambda = \pm 1$. 
There are no other isomorphims between the
irreducible representations than the ones coming from the previous list.
In particular, any irreducible representation is a quotient of an
$I(\sigma_r)\otimes\left(\chi\circ\det\right)$ for some $r$ and $\chi$, and $r$ and
$\chi$ are almost entirely determined by the quotient itself.

We recall the following result:

\begin{prop}
\label{quotientI}
Let $Q$ be a finite length quotient of $I(\sigma_a(b))$. Then the
Jordan-Hoelder factors of $Q$ are of the form $\pi(a,\lambda,\omega^b)$
for some $\lambda\in \bar\F_p$ if $a \not\in \{0,p-1\}$. If $a \in \{0,p-1\}$, the
Jordan-Hoelder factors of $Q$ are of the form $\pi(a,\lambda,\omega^b)$
for $\lambda \not\in \{\pm 1\}$, or $\St\otimes\left(\omega^b\unr(\pm 1)\circ
\det\right)$ or $\left(\omega^b\unr(\pm 1)\circ \det\right)$.
\end{prop}

Here is a slightly more precise statement, which follows immediatly the
results of \cite{BL94,BL95,Br03a} although it is not stated
explicitely there:

\begin{prop}
\label{quotientIirr}
Let $Q$ be an irreducible representation of $G$, and $\pi$ a nonzero map 
$I(\sigma_a(b)) \to Q$. Then there exists a unique $\lambda$ such that
$\pi$ is zero on $(T-\lambda)I(\sigma_a(b))$, and then $Q$ is isomorphic to
$\pi(a,\lambda,\omega^b)$ if $(a,\lambda)\not\in \{(0,\pm 1),(p-1,\pm
1)\}$, or $Q$ is isomorphic to $\St \otimes
\left(\omega^b\unr(\lambda)\circ\det\right)$ if $a = p-1$ and $\lambda = \pm 1$, or
$Q$ is isomorphic to 
$\left(\omega^b\unr(\lambda)\circ\det\right)$ if $a = 0$ and $\lambda = \pm 1$.
\end{prop}

The point is that any non-zero map $I(\sigma_a(b)) \to Q$ where $Q$ is
irreducible is (a scalar multiple of) the obvious map, and not something else.

\begin{proof}[Proof of Proposition \ref{quotientIirr}]
By Frobenius reciprocity we have: 
$$\Hom_G(I(\sigma_a(b)),Q) =
\Hom_K(\sigma_a(b),Q|_K) = \Hom_K(\sigma_a(b),\soc_K(Q))$$ 
But we know that for any irreducible $Q$, every representation that
appears in $\soc_K(Q)$ appears with multiplicity one (this is a standard
computation for the principal series case, and 
\cite[Théorème 3.2.4]{Br03a} for the supersingular case). So the
dimension of $\Hom_G(I(\sigma_a(b)),Q)$ is at most one.  
\end{proof}

\subsection{Definition over a subfield}

Let $E$ be a finite extension of $\F_p$. 
Let $\pi$ be a representation of $G$ with coefficients in
$\bar\F_p$. We say that $\pi$ has a model over $E$ if there is a 
representation $\pi'$ of $G$ with coefficients in $E$ such that $\pi$ is
isomorphic to $\pi'\otimes_E\bar\F_p$.
It follows from the definition
of $\pi(r,\lambda,1)$ in \cite{BL94} that:

\begin{prop}
\label{lambdainE}
The representation $\pi(r,\lambda,1)$ has a model over $E$ if and only if $\lambda\in E$.
\end{prop}

\section{Modulo $p$ and $p$-adic Langlands correspondence}
\label{correspondences}

\subsection{The correspondence modulo $p$}

The semi-simple modulo $p$ Langlands correspondence, defined in
\cite{Br03a}, relates semi-simple
representations of dimension $2$ of $G_{\Q_p}$ over $\bar\F_p$ with
some finite length smooth admissible semi-simple representations of $G$
over $\bar\F_p$. 
Note that semi-simple representations of dimension $2$ of $G_{\Q_p}$ are
of two kinds: either they are reducible, or they are isomorphic to some
$(\ind\omega_2^a)\otimes \chi$ with $1 \leq a \leq p$, 
where $\ind\omega_2^a$ is the unique
representation with determinant $\omega^a$ and restriction to inertia
isomorphic to $\omega_2^a\oplus \omega_2^{pa}$.

More precisely, if $V$ is a representation of
dimension $2$ of $G_{\Q_p}$ with coefficients in $\bar\F_p$, we denote by
$\Pi(V)$ the attached representation of $G$:

\begin{theo}[\cite{Br03a}, Définition 4.2.4]
\label{descrcorresp}
The semi-simple modulo $p$ Langlands correspondence is as follows:
\begin{itemize}
\item
If $V = (\ind\omega_2^{r+1})\otimes \chi$ for $0 \leq r \leq p-1$ then
$\Pi(V) = \pi(r,0,\chi)$.
\item
If $V = (\unr(\lambda)\omega^{r+1}\oplus\unr(\lambda^{-1}))\otimes\chi$  
for $0 \leq r \leq p-1$ then 
$$\Pi(V) =
\pi(r,\lambda,\chi)^{ss}\oplus\pi([p-3-r],\lambda^{-1},\omega^{r+1}\chi)^{ss}$$
\end{itemize}
\end{theo}

\subsection{Rationality questions}
\label{rationality}

In fact, the semi-simple modulo $p$ Langlands correspondences attaches
representation of $G$ with coefficients in $E$ to representations of
$G_{\Q_p}$ defined over $E$, but their irreducible components are not
necessarily defined over $E$. In order to do explicit computations, we
need to determine over which field to work.

Let $\bar\rho$ a $2$-dimensional semi-simple 
representation of $G_{\Q_p}$ over $\bar\F_p$,
with determinant equal to a power of
the cyclotomic character $\omega$.
If $\bar\rho$ is irreducible, it has a model over $\F_p$. Otherwise,
it is isomorphic to some $\omega^a\unr(\lambda) \oplus
\omega^b\unr(\lambda^{-1})$.

\begin{prop}
\label{ratGalois}
Let $\lambda\in \bar\F_p^\times$.
Suppose that 
$\omega^a\unr(\lambda) \oplus \omega^b\unr(\lambda^{-1})$
has a model over $E$. 
Then $\lambda + \lambda^{-1} \in E$.
If $a \not\equiv b$ modulo $p-1$, then $\lambda \in E$.
\end{prop}

\begin{proof}
Consider the traces of the image by $\bar\rho$ of some elements of $G_{\Q_p}$.
\end{proof}

The analogue of Proposition \ref{ratGalois} on the side of
representations of $G$ is:

\begin{prop}
\label{ratGL2}
Let $\lambda\in \bar\F_p^\times$, and $\chi$ a power of $\omega$.
Suppose that
$\pi(r,\lambda,\chi)^{ss}\oplus\pi([p-3-r],\lambda^{-1},\omega^{r+1}\chi)^{ss}$
is defined over $E$.
Then $\lambda + \lambda^{-1} \in E$.

If $p > 2$ and $r \neq p-2$, then $\lambda \in E$, and in particular each
Jordan-Hoelder factor of 
$\pi(r,\lambda,\chi)^{ss}\oplus\pi([p-3-r],\lambda^{-1},\omega^{r+1}\chi)^{ss}$
is defined over $E$.

Moreover, if $\lambda\not\in E$ then 
$\pi(r,\lambda,\chi)\oplus\pi([p-3-r],\lambda^{-1},\omega^{r+1}\chi)$
is irreducible over $E$.
\end{prop}

\begin{proof}
We can deduce the first two parts from Proposition \ref{ratGalois} via
the semi-simple Langlands
correspondence modulo $p$ (or we could prove it directly).

Suppose now that
$\lambda \neq \pm 1$, 
so that $\pi(r,\lambda,\chi)$ and $\pi([p-3-r],\lambda^{-1},\omega^{r+1}\chi)$
are both irreducible. Suppose that 
$\pi(r,\lambda,\chi)\oplus\pi([p-3-r],\lambda^{-1},\omega^{r+1}\chi)$
is reducible over $E$, then 
$\pi(r,\lambda,\chi)$ and $\pi([p-3-r],\lambda^{-1},\omega^{r+1}\chi)$
are both defined over $E$, so by Proposition \ref{lambdainE} we have that
$\lambda\in E$.
\end{proof}

\subsection{Jordan-Hoelder factors}

If $\Pi$ is a finite length representation of $G$ defined over $E$, we denote by
$\JH(\Pi)$ the set of the irreducible representations that appear as
its Jordan-Hoelder factors (taking into account multiplicities) as a
representation with coefficients in $E$.

Consider a $2$-dimensional Galois representation $V$ defined over $E$
such that $\det V$ is a power of the cyclotomic character, and $\Pi(V)$ the smooth
representation of $\GL_2(\Q_p)$ attached to it by the semi-simple
Langlands correspondence modulo $p$. Then $\JH(\Pi(V))$ is a finite set
of irreducible representations of $\GL_2(\Q_p)$ with coefficients in $E$.
It follows from Theorem \ref{descrcorresp} that not all finite set of
irreducible representations can appear in this way.
Denote by $\jh$ the set of sets of irreducible representations 
that can appear in this way.
%
%
We deduce from Propositions \ref{descrcorresp} and \ref{ratGL2} the
following description of $\jh$:

\begin{prop}
\label{JHTheta}
In what follows, $\chi$ always denotes a power of the cyclomotic
character.
$\jh$ contains exactly the following sets:
\begin{enumerate}
\item
The singleton $\{\pi(r,0,\chi)\}$ for any $r$ and any $\chi$

\item
The set
$\{\pi(r,\lambda,\chi),\pi([p-3-r],\lambda^{-1},\chi\omega^{r+1})\}$
for any $\chi$, any $\lambda\in E^\times$ and any $r$ such that
either $r$ is different from $0$, $p-3$, $p-1$ or $\lambda\not\in \{\pm 1\}$.

\item
The singleton
$\{\pi(p-2,\lambda,\chi)\oplus\pi(p-2,\lambda^{-1},\chi)\}$ for 
any $\chi$, and any $\lambda\in \bar\F_p\setminus E$ such that
$\lambda+\lambda^{-1} \in E$.

\item
For $p > 3$ and $\lambda = \pm 1$ and any $\chi$, 
the set $\{\St\otimes \left(\unr(\lambda)\chi \circ \det\right)$, 
$\left(\unr(\lambda)\chi\circ\det\right)$, $\pi(p-3,\lambda^{-1},\chi\omega)\}$.

\item
For $p=3$ and $\lambda = \pm 1$ and any $\chi$, the set
$\{\St\otimes \left(\unr(\lambda)\chi \circ \det\right)$,
$\left(\unr(\lambda)\chi\circ\det\right)$,
$\St\otimes \left(\unr(\lambda)\chi\omega \circ \det\right)$,
$\left(\unr(\lambda)\chi\omega\circ\det\right)\}$.

\item
For $p=2$ and any $\chi$, the set
$\{\St\otimes \left(\chi \circ \det\right)\text{ twice},
\left(\chi\circ\det\right)\text{ twice}\}$
\end{enumerate}
\end{prop}

We see that $r = p-2$ plays a special role, as in this case $r =
[p-3-r]$ and $\omega^{r+1} = 1$.
Moreover, for $p=2$, $r = 1 = p-1$ also plays a special role as
for all $\lambda$, $\JH(\pi(0,\lambda,\chi)) =
\JH(\pi(p-1,\lambda,\chi))$ so $0$ and $p-1$ play similar roles.

Note that in case (2) the set can be twice the same representation,
when 
$r = p-2$ and $\lambda = \pm 1$. Except in this case and case (6),
the sets are in fact made of distinct representations.

We note that two distinct elements of $\jh$ are disjoint sets.
We deduce immediatly the following important consequence:
\begin{prop}
\label{uniquejh}
$\JH(\Pi(V))$ is uniquely determined by any of the Jordan-Hoelder factors
that is contains.
\end{prop}

We can also extract from Proposition \ref{JHTheta} the following remarks:

\begin{fact}
\label{factsjh}
\begin{enumerate}
\item
For any $(r,\lambda,\chi)$ there exists a unique $S \in \jh$
containing one Jordan-Hoelder factor of $\pi(r,\lambda,\chi)$, and then
it contains all of them.

\item
For any $\mu\in E$, all the Jordan-Hoelder factors of
$I(\sigma_{p-2}(b)_E)/(T^2-\mu T+1)$ are in the same $S \in \jh$.

\item
If $p = 2$  and any $\mu\in E$ and $i \in \{ 0,1 \}$, all the Jordan-Hoelder factors of 
$I(\sigma_{i,E})/(T^2-\mu T+1)$ are in the same $S \in \jh$.
\end{enumerate}
\end{fact}

\subsection{The $p$-adic correspondence in the crystalline case}

\subsubsection{The correspondence}
\label{correspondence}

The $p$-adic Langlands correspondence attaches to a representation $V$ of
dimension $2$ of $G_{\Q_p}$ over some finite extension of $\Q_p$ a
unitary Banach representation $\Pi(V)$ of $G$ over the same field. In the case
where the Galois representation is irreducible and crystalline with
distinct Hodge-Tate weights, this unitary Banach
representation has an explicit description, as was conjectured in
\cite{Br03b} and proved in \cite{BB10}.
Let $V_{k,a_p}$ be such a crystalline representation. We denote
$\Pi(V_{k,a_p})$ by $\Pi_{k,a_p}$. 
We also denote by $\bar\Pi_{k,a_p}$ the semi-simplification of the
reduction modulo $p$ with respect
to any $G$-invariant, finite-type lattice in $\Pi_{k,a_p}$
Recall that a lattice in a locally algebraic representation $V$ of $G$ over $\K$ is
a sub-$R[G]$-module of $V$ that generates $V$ over $\K$, and
that does not contain a $\K$-line. 

Moreover, this correspondence is compatible with reduction modulo $p$ by
\cite{Ber10}. This means that $\bar\Pi_{k,a_p}$ is the
semi-simple representation of $G$ attached to $\bar{V}_{k,a_p}^{ss}$ by
the semi-simple Langlands correspondence modulo $p$. So the computation
of $\bar\Pi_{k,a_p}$ is equivalent to the computation of
$\bar{V}_{k,a_p}^{ss}$.

\subsubsection{Description of $\Pi_{k,a_p}$}

We can describe $\Pi_{k,a_p}$ as follows. 
Let $r = k-2$.

We take here $\K$ any finite extension of $\Q_p$ containing $a_p$, with ring of
integers $R$ and residue field $E$, so that the Galois
representation $V_{k,a_p}$ has a model over $\K$. 
We denote by $\Theta_{k,a_p}$ the image of $I(\Sym^r R^2) \subset
I(\Sym^r \K^2)$ in the quotient $I(\Sym^r \K^2)/(T-a_p)I(\Sym^r \K^2)$.
As $I(\Sym^r R^2)$ is a sub-$R[G]$-module of $I(\Sym^r \K^2)$, 
$\Theta_{k,a_p}$ is a sub-$R[G]$-module of 
$I(\Sym^r \K^2)/(T-a_p)I(\Sym^r \K^2)$.

We now recall the following result:

\begin{prop}
\label{descrcorresp0}
$\Theta_{k,a_p}$ is a lattice in $I(\Sym^r \K^2)/(T-a_p)I(\Sym^r \K^2)$,
and $\Pi_{k,a_p}$ is the completion of $I(\Sym^r \K^2)/(T-a_p)I(\Sym^r \K^2)$
with respect to this lattice. 
\end{prop}

We sketch a proof of Proposition \ref{descrcorresp0}, as we could not
find a reference that explains what happens in the case of 
$a_p \in \{\pm (1+p)p^{k/2-1}\}$.

\begin{proof}
This is essentially the main result of \cite{BB10}. In this article, the
authors attach to $k$ and $a_p$ a locally algebraic representation of $G$
over $\K$, which they denote by $\pi(\alpha)$ and they show
that there exists a unique commensurability class of lattices in
$\pi(\alpha)$ (\cite[Corollaire 5.3.4]{BB10}). Then $\Pi_{k,a_p}$ is
defined to be the completion of $\pi(\alpha)$ with respect to any lattice
of this class.

So we must show that the representation $\pi(\alpha)$ of \cite{BB10} is
isomorphic to $I(\Sym^r \K^2)/(T-a_p)I(\Sym^r \K^2)$,
and that $\Theta_{k,a_p}$ is a lattice in this representation. This will
show that both definitions of $\Pi_{k,a_p}$ coincide and prove
Proposition \ref{descrcorresp0}.

The fact that $\pi(\alpha)$ is isomorphic to 
$I(\Sym^r \K^2)/(T-a_p)I(\Sym^r \K^2)$ is the consequence of \cite[Proposition 3.2.1]{Br03b}.
It follows directly from the proposition when $a_p \not\in \{\pm
(1+p)p^{k/2-1}\}$. When $a_p \in \{\pm (1+p)p^{k/2-1}\}$, then we notice
that both $\pi(\alpha)$ and $I(\Sym^r \K^2)/(T-a_p)I(\Sym^r \K^2)$ are of
the form $\Sym^{k-2}\K^2 \otimes \left(\unr((p+1)/a_p)\circ\det\right) \otimes W$
where $W$ is a smooth representation of $G$ which is a non-split
extension of the trivial representation by the Steinberg representation.
By the results of \cite{SS91} there is only one isomorphism class of such
representations.

Finally, we need to see that $\Theta_{k,a_p}$ is a lattice. 
The only part that is not trivial is the fact that it does not contain a
$\K$-line.
This is proved directly in \cite{Br03b} for small values of $k$. In
general this a consequence of \cite[Corollaire 5.3.4]{BB10}: it is enough
to see that $\Theta_{k,a_p}$ is contained in a lattice, but as
$\Theta_{k,a_p}$ is of the form $R[G]x$ for an $x\in \pi(\alpha)$, this
is true as soon as we know that there exists at least one lattice in
$\pi(\alpha)$.
\end{proof}

\begin{coro}
$\bar\Pi_{k,a_p}$ is equal to $\bar\Theta_{k,a_p}^{ss}$. 
\end{coro}

\subsubsection{Notation}

Let $\M_{k,a_p}$ be the submodule
$(T-a_p)I(\Sym^r \K^2) \cap I(\Sym^r R^2)$ of $I(\Sym^r \K^2)$. It is
$G$-invariant and is a free submodule of 
$I(\Sym^r R^2)$, 
and $\Theta_{k,a_p} = I(\Sym^r R^2) / \M_{k,a_p}$.
We denote by $M_{k,a_p}$ the image in $I(\Sym^r E^2)$ of $\M_{k,a_p}$
so that $\bar\Theta_{k,a_p} = I(\Sym^r E^2) /M_{k,a_p}$.

\subsection{Identifying $\bar{V}_{k,a_p}$ from Jordan-Hoelder factors
of $\bar\Theta_{k,a_p}$}

\begin{prop}
\label{fromjh}
Let $V$ be a representation of $G_{\Q_p}$ of dimension $2$. Let
$\bar\Pi(V)$ be the reduction modulo $p$ of $\Pi(V)$ with respect to some
$G$-invariant lattice. Suppose that
we have a surjection $\Pi' \to \bar\Pi(V)$ where $\Pi'$ is a finite length
representation of $G$ over $E$, and suppose that
there exists a unique $S \in \jh$ such that $S \subset \JH(\Pi')$. Then
$\JH(\bar\Pi(V)) = S$.
\end{prop}

In the situation that interests us, $\Pi(V) = \bar\Theta_{k,a_p} = 
I(\sigma_{k-2,E})/M_{k,a_p}$, and we will try to find a
subspace $M'$ of $M_{k,a_p}$ with $\Pi'$ the quotient of
$I(\sigma_{k-2,E})$ by the subspace generated by $M'$.

\section{Computing the reduction from a set of relations}
\label{computefromrel}

In this Section we write $\bar\Theta$ for $\bar\Theta_{k,a_p}$ and $r = k-2$.
We will explain how to obtain information on 
$\bar\Theta$ as a representation of $G$, and more precisely about
$\JH(\bar\Theta)$, by doing only linear algebra computations using
subspaces of finite dimension of $M_{k,a_p}$. We denote by $E$ a finite
field containing the residue field of $\Q_p(a_p)$.

\subsection{Use of a filtration}

Suppose that we are given a decreasing filtration $(V_i)$ of $\sigma_{r,E}$, with $V_0
= \sigma_{r,E}$ and such that $V_i/V_{i+1} = J_i$ is an irreducible representation of
$K$. Then $J_i$ is isomorphic to $\sigma_{a_i}(b_i)_E$ for some $a_i \in
\{0\dots,p-1\}$ and some $b_i \in \Z/(p-1)\Z$.
We describe such a filtration in Section \ref{filtration}. 

Let $\Lambda = I(\sigma_{r,E})$. Then we have a filtration $(\Lambda_i)$ 
of sub-$G$-representations by setting $\Lambda_i = I(V_i)$. Let 
$Q_i = \Lambda_i/\Lambda_{i+1} = I(J_i)$. It is endowed with a Hecke
operator $T_i$ as in Paragraph \ref{compactind}. 

From the filtration $(\Lambda_i)$ we get a filtration $(\pi(\Lambda_i))$
on $\bar\Theta$, and we let $F_i = \pi(\Lambda_i)/\pi(\Lambda_{i+1})$, so
that $F_i$ is a quotient of $Q_i$. Let $\pi_i$ be the map $Q_i \to F_i$.
Note that each $F_i$ is defined over $E$.

\begin{theo}
\label{factor}
For each $i$, at least one of the following is true:
\begin{enumerate}
\item 
$F_i = 0$.

\item
$\pi_i((T_i-\lambda)Q_i) = 0$ for some $\lambda\in E$.

\item
$\pi_i((T_i^2-\mu T_i + 1)Q_i) = 0$ for some $\mu \in E$.
\end{enumerate}
Moreover, we can be in case (3) but not (2) only for $a_i = p-2$, or
$p=2$.
\end{theo}

\begin{proof}
Each $F_i$ is of finite length, and $\JH(\bar\Theta) = \cup_i \JH(F_i)$.

From Proposition \ref{JHTheta}, we see that if $p > 2$, each $F_i$ is of length at
most $2$. If $p = 2$ then $F_i$ is of length at most $4$.

Suppose first that $F_i$ is irreducible. Then it is a quotient of
$Q_i/(T_i-\lambda)Q_i$ for some $\lambda\in E$, by Proposition
\ref{quotientIirr}. So we are in case (2).

Suppose now that $F_i$ has length at least $2$. It can happen only in the cases
where $a_i$ is $0$, $p-1$ or $p-2$. We treat these cases in the following
lemmas, and we see that:
either $p > 2$ and $a_i = 0$ or $p-1$, and then we are in case (2), or 
$a_i = p-2$ and we are in case (3), or $p = 2$, and then we are in case
(2) or (3) depending on the length of $F_i$.

As each $F_i$ is defined over $E$, then by Proposition
\ref{lambdainE}, if we are in case (2) then $\lambda\in E$. If we are in
case (3) but not (2), then $\mu\in E$. In particular, if the set of
Jordan-Hoelder factors is as in (3) of Proposition \ref{JHTheta}, then
there exists only one $i$ such that $F_i \neq 0$ and it is in the case
(3) but not (2).
\end{proof}

\begin{lemm}
Suppose $p > 2$.
Let $F$ be a quotient of $I(\sigma_{p-2})$ and $\JH(F) =
\{\pi(p-2,\lambda,1),\pi(p-2,\lambda^{-1},1)\}$. 
Then the given map $I(\sigma_{p-2}) \to F$ is zero 
on $(T^2-(\lambda+\lambda^{-1})T+1)I(\sigma_{p-2})$ and induces an
isomorphism from $I(\sigma_{p-2})/(T^2-(\lambda+\lambda^{-1})T+1)$ to
$F$.
\end{lemm}

\begin{proof}
Suppose we have a surjective map $\pi : F \to \pi(p-2,\lambda,1)$, and
let $\alpha$ be the given surjection $I(\sigma_{p-2}) \to F$. Then
$\pi\circ\alpha$ is zero on $(T-\lambda)I(\sigma_{p-2})$ by Proposition
\ref{quotientIirr}. Let $\beta = \alpha \circ (T-\lambda)$. Then $\im
\beta \subset \ker\pi$, and in fact $\im \beta = \ker\pi$ otherwise
$\beta = 0$ and $F = I(\sigma_{p-2})/(T-\lambda)$. So $\im\beta$ is
isomorphic to $\pi(p-2,\lambda^{-1},1)$, and $\beta$ is zero on
$(T-\lambda)^{-1}I(\sigma_{p-2})$. So $\alpha$ is zero on 
$(T^2-(\lambda+\lambda^{-1})T+1)I(\sigma_{p-2})$, and $F$ is a quotient 
of $I(\sigma_{p-2})/(T^2-(\lambda+\lambda^{-1})T+1)$, and so is equal to
it.
\end{proof}

\begin{lemm}
Suppose $p = 2$, $i = 0$ or $1$ and $F$ is a quotient of $I(\sigma_i)$
with $\JH(F) \subset \{1,1,\St,\St\}$. Then the map $\alpha : I(\sigma_i)
\to F$ is zero on $(T^2+1)I(\sigma_i)$.
\end{lemm}

\begin{proof}
Suppose $i = 0$. Only $\St$ can be a quotient of $I(\sigma_0)$, so there
exists a surjective map $\pi : F \to \St$. Then by Proposition
\ref{quotientIirr}, $\pi\circ \alpha$ is zero on $(T-1)I(\sigma_0)$. Let
$\beta = \alpha \circ (T-1)$. Let $F' = \im \beta$, then $F' \subset
\ker\pi$ so $\JH(F') \subset \{1,1,\St\}$. As before, there exists a
surjection $\pi' : F' \to \St$ and so a surjection $\pi' \circ \beta :
I(\sigma_0) \to \St$, so $\pi' \circ \beta$ is zero on $(T-1)I(\sigma_0)$. 
Let $\gamma = \beta \circ (T-1)$, then $\im \gamma \subset \ker \pi'$, so
$\JH(\im \gamma) \subset \{1,1\}$, and so in fact $\gamma = 0$. So
$\alpha \circ (T-1)^2$ is zero.

For $i = 1$ the proof is the same but with the role of $1$ and $\St$
reversed.
\end{proof}

We state the following lemmas without proof, as their proofs are very
similar to the previous ones:

\begin{lemm}
Let $F$ be a quotient of $I(\sigma_0)$ and $\JH(F) =  \{\St,1\}$.
Then $F$ is the quotient of $I(\sigma_0)$ by $(T-1)$.
\end{lemm}

\begin{lemm}
Let $F$ be a quotient of $I(\sigma_{p-1})$ and $\JH(F) =  \{\St,1\}$.
Then $F$ is the quotient of $I(\sigma_{p-1})$ by $(T-1)$.
\end{lemm}

\begin{lemm}
Suppose $p = 2$, $i = 0$ or $1$ and $F$ is a quotient of $I(\sigma_i)$
with $\JH(F) = \{\pi(0,\lambda,1),\pi(0,\lambda^{-1},1)\}$ for some 
$\lambda \not\in \{0,1\}$. Then $F$ is the quotient of $I(\sigma_i)$
by $(T^2-(\lambda + \lambda^{-1})T+1)I(\sigma_i)$.
\end{lemm}

\subsection{Computing $\JH(\bar\Theta)$}
\label{explicitF}

Suppose now that in addition to the filtration, we are given for each $i$
an element $v_i \in \Lambda_i$ such that the image of $v_i$ in $Q_i$
generates it as a $G$-representation, and elements $w_{1,i}$ and
$w_{2,i}$ in $\Lambda_i$ whose image in $Q_i$ is equal to $T_iv_i$ and
$T_i^2v_i$ respectively. 

Suppose moreovoer that we are given a subspace $M'$ of $M_{k,a_p}$. Then we can
answer the following questions:
\begin{enumerate}
\item
is $v_i$ in $M' + \Lambda_{i+1}$ ? 

\item
is $w_{1,i}-\lambda v_i$ in $M' + \Lambda_{i+1}$ for some $\lambda\in E$ ? 

\item
is $w_{2,i}-\mu w_{1,i} + 1$ in $M' + \Lambda_{i+1}$ for some $\mu \in E$ ? 
\end{enumerate}

If the answer to the first question is positive, then $F_i = 0$. If the
answer to the second question is positive, then $F_i$
is a quotient of $Q_i/(T_i-\lambda)Q_i$ for the given $\lambda$. If the
answer to the third question is positive, then $F_i$ is a quotient of
$Q_i/(T_i^2-\mu T_i +1)Q_i$ for the given $\mu$.

\begin{prop}
\label{cancompute}
Let $M'$ be a finite-dimensional subspace of $M_{k,a_p}$.
Suppose that: for each $i$, we can answer positively to one of the three
questions above, and moreover, the Jordan-Hoelder factors of the
representations
$Q_i/(T-\lambda)Q_i$ or $Q_i/(T^2-\mu T +1)Q_i$ that appear are all in the same $S
\in \jh$.
Then $M'$ contains enough information to compute $\JH(\bar\Theta)$.
\end{prop}

\begin{proof}
This is a reformulation of Proposition \ref{fromjh}, as $\bar\Theta$ is a
quotient of $I(\sigma_{r,E})$ by $E[G]M'$.
\end{proof}

\begin{coro}
\label{finitedimenough}
There exists a finite-dimensional subspace 
$M'$ of $M_{k,a_p}$ that contains enough information
to determine the value of $\JH(\bar\Theta)$.
\end{coro}

\begin{rema}
This does not mean that for this $M'$ we know exactly which $F_i$
contributes which Jordan-Hoelder factor of $\bar\Theta$.
For example in the case where $Q_i/(T_i-\lambda)Q_i$ is reducible, we are not able to say
using only this information
if $F_i$ is equal to all of $Q_i/(T_i-\lambda)Q_i$ or if it is a strict
quotient of it. However this is not
necessary to know this to determine $\JH(\bar\Theta)$.
\end{rema}

\section{Description of the filtration on $\Sym^{k-2} E^2$}
\label{filtration}

We describe an explicit filtration on $\Sym^r E^2$ that can be used in
the computations of the previous section, which has interesting
properties with respect to these computations. We will call this
filtration the standard filtration.

We identify $\Sym^r E^2$ with the space $E[X,Y]_r$ of homogeneous
polynomials of degree $r$ with coefficients in $E$.
Let $\theta = X^pY-XY^p$.

Let $m = \lfloor r/(p+1) \rfloor$.
If $u < m$, we set $\eta_{u,v} = \theta^uX^vY^{r-u(p+1)-v}$ if $0 \leq v
\leq p-1$, and $\eta_{u,p} = \theta^uX^{r-u(p+1)}$. If $0 \leq v \leq
r-m(p+1)$ we set $\eta_{m,v} = \theta^mX^vY^{r-m(p+1)-v}$.

\subsection{Start of the filtration}

\subsubsection{Subspace generated by $\theta$}

Let $V_0 = E[X,Y]_r$.

We see easily that for all $g \in
K$, we have $g\cdot \theta = \det(g)\theta$.
So the subspace of $V_0$ generated by $\theta$ is invariant under the
action of $K$, we call it $V_2$.

Let us give a basis of $V_0/V_2$. It is a vector space of dimension
$p+1$. 
As $X^iY^{r-i} = X^{p-1+i}Y^{r-i-p+1}$ mod $\theta$,
any monomial in $E[X,Y]_r$ is equivalent modulo $\theta$ to one of the
elements
$X^r$, $Y^r$ or $X^iY^{r-i}$ for $1\leq i \leq p-1$.  
So we get that a basis of $V_0/V_2$ is given by the
$\eta_{0,v}$ for $0 \leq v \leq p$ (or $v \leq r$ if $r \leq p$).

\subsubsection{Refining the filtration}

Let $V_1$ the subspace of $V_0$ generated as an $E[K]$-module by $V_2$ and $X^r$.  
Let $J_0 = V_0/V_1$ and $J_1 = V_1/V_2$.
The image of $X^r$ under the action of $K$ is the set of elements of the
form $(uX+vY)^r$ for  $u,v\in \F_p$ not both zero, so these elements
generate $J_1$ as an $E$-vector space.
Another generating set is also given by $X^r$, $Y^r$ and the elements
$\s_i$ for $0\leq i \leq p-2$, where $\s_i =
\sum_{j\geq 0}\binom{r}{j(p-1)+i}X^{j(p-1)+i}Y^{r-j(p-1)-i}$ (some of these sums
can be zero modulo $V_2$). The correspondence between these two generating sets is
given by the equalities $(uX+Y)^r = \sum_{i=0}^{p-2}u^i\s_i$ for $u\in
\F_p^\times$, and $\s_i = -\sum_{u\in \F_p^\times}u^{-i}(uX+Y)^r$.

\begin{lemm}
\label{summodp}
Let $0 \leq i \leq p-2$. Then
$\s_{i,r} := \sum_j\binom{r}{j(p-1)+i}$ is equal to $\binom{[r]}{i}$ modulo $p$, 
except in the case $[r]=p-1$ and $i=0$ where the sum equals $2$ modulo $p$.
\end{lemm}

\begin{proof}
Let $f(X) = X^{-i}(1+X)^r = \sum_{j\geq i}\binom{r}{j}X^{j-i}$.
We have 
$$\sum_{x\in\F_p^\times}f(x) = -\sum_{j=i\text{ mod }p-1}\binom{r}{j}
= -\sum_{j\geq 0}\binom{r}{j(p-1)+i} = -\s_{i,r}$$

Moreover the value of $f(x)$ for $x\in\F_p^\times$ depends only on
$r$ modulo $p-1$, that is $[r]$. So
we only need to compute the sum when $r=[r]$, and then there is either only one
term in the sum, or two terms when $[r]=p-1$ and $i=0$.
\end{proof}

\begin{coro}
$J_1$ is of dimension $[r] + 1$, and a basis of $J_1$ is given by the
elements $\eta_{0,v}$ for $0 \leq v < [r]$ and $v = p$.  A basis of $J_0$
is given by the elements $\eta_{0,v}$ for $[r] \leq v \leq p-1$.
\end{coro}

\begin{proof}
Modulo $V_2$, $X^{j(p-1)+i}Y^{r-j(p-1)+i}$ and
$X^{j'(p-1)+i}Y^{r-j'(p-1)+i}$ are equal, except when one of these terms is either
$X^r$ or $Y^r$. So Lemma \ref{summodp} allows us to compute $\s_i$
modulo $V_2$, which gives the results for $J_1$.
\end{proof}

\begin{prop}
The representation $J_1$ is isomorphic to $\sigma_{[r]}$, and $J_0$ to
$\sigma_{p-1-[r]}([r])$. This extension is split if and only if $[r] = p-1$. 
\end{prop}

\begin{proof}
It is well known that $V_0/V_2$ is an extension of $\sigma_{p-1-[r]}([r])$ by
$\sigma_{[r]}$ (see for example \cite{Glo78}). The fact that this extension is split
if and only if $[r] = p-1$ is a consequence of Proposition \ref{genvector} below. 
So $J_1$ is the unique subrepresentation of
dimension $[r]+1$ of $V_0/V_2$, which is $\sigma_{[r]}$.
\end{proof}

We write $a_0 = p-1-[r]$, $a_1 = [r]$, $b_0 = [r]$, $b_1 = 0$, so that
$J_i$ is isomorphic to $\sigma_{a_i}(b_i)$ for $i = 0,1$.

\subsubsection{Generating elements}

A simple computation shows:

\begin{prop}
\label{genvector}
Let $\sigma_a(b)$ be identified with the vector space $E[x,y]_a$.
Then we can recognize the line $Ex^a$ by the action of $\GL_2(\F_p)$ on 
$E[x,y]_a$:
If $ a < p-1$, then $Ex^a$ is the unique line of elements on which the matrix
$\smatr u00v$ acts by $u^{a+b}v^b$, and moreover $x^a$ is invariant by
the matrix $\smatr 1101$. If $a = p-1$, $Ex^{p-1}$ is the unique line on
which the matrix $\smatr u00v$ acts by $u^bv^b$ and which is invariant by
the matrix $\smatr 1101$.
\end{prop}

We set $e_0 = \eta_{0,p-1}$ and $e_1 = \eta_{0,p}$. 

Then for $i = 0,1$, $e_i$ generates $J_i$ as a $K$-representation.
Moreover, in the isomorphism $J_i = \sigma_{a_i}(b_i) = E[x,y]_{a_i}$, 
the image of $e_i$ corresponds to $x^{a_i}$ (up to some scalar) by
Proposition \ref{genvector}.

\subsection{Filtration by powers of $\theta$}

For all $0 \leq i \leq m$, the multiplication by $\theta^i$ induces a
linear map $E[X,Y]_{r-i(p+1)} \to E[X,Y]_r$, and the image is a
subrepresentation isomorphic to $\sigma_{r-i(p+1)}(i)_E$. We denote it by
$V_{2i}$.

By considering the analogue of $V_1$ for $E[X,Y]_{r-i(p+1)}$, and considering
the multiplication by $\theta^i$ from $E[X,Y]_{r-i(p+1)}$ to $E[X,Y]_r$ for $i
\leq m$,
we refine the filtration
$(V_{2i})$ into a filtration 
$(V_i)$. Let $J_i = V_i/V_{i+1}$.

We set $a_{2i+1} = [r-2i]$, and $b_{2i+1} = i$, 
and $a_{2i} = p-1-[r-2i]$, and $b_{2i} = r-i$.
We get that $J_i$ is isomorphic to $\sigma_{a_i}(b_i)$.

\subsection{End of the filtration}

The subspace $V_{2m}$ is the smallest we can obtain via powers of
$\theta$, but it is not necessarily irreducible.
Let $t$ be the remainder of the Euclidean division of $r$ by $p+1$, so
that $r = m(p+1)+t$. Then $V_{2m} = E_t(m)$.

If $t \leq p-1$ then $V_{2m}$ is irreducible. We set $V_{2m+1} = 0$,
$a_{2m} = t$, $b_{2m} = m$.

If $t = p$ then $V_{2m}$ is reducible. We denote by $V_{2m+1}$
the subrepresentation generated by $\theta^mX^p$. Let $V_{2m+2} = 0$.
Set $a_{2m} = p-2$, $b_{2m} = m+1$, $a_{2m+1} = 1$, $b_{2m+1} = m$.

With these notations we have that $J_i$ is isomorphic to
$\sigma_{a_i}(b_i)$ as before.

\subsection{Basis and generators}
\label{explicitbasis}

We summarize here the explicit description of the elements of the
filtration.

\subsubsection{A basis of $J_i$}

The description of a basis of $J_0$ and $J_1$ generalizes to all $i$:
If $i < m$, 
$J_{2i}$ has a basis given by the $\eta_{i,v}$ for $[r-2i] \leq v \leq
p-1$, and $J_{2i+1}$ has a basis given by the $\eta_{i,v}$ for $0 \leq v
< [r-2i]$ and $v = p$. 

If $V_{2m+1} = 0$, then $J_{2m}$ has a basis given by the $\eta_{m,v}$,
$0 \leq v \leq r-m(p+1)$.

If $r = m(p+1) + p$ then $J_{2m}$ has a basis given by $\eta_{m,v}$ 
for $1 \leq v \leq p-1$, and $J_{2m+1}$ has a basis given by $\eta_{m,0}$
and $\eta_{m,p}$.

\subsubsection{A generating element}

If $i < m$ we set $e_{2i} = \eta_{i,p-1}$ and $e_{2i+1} = \eta_{i,p}$.

If $V_{2m+1} = 0$ we set $e_{2m} =  \eta_{m,t}$.

If $V_{2m+1} \neq 0$ we set
$e_{2m} = \eta_{m,p-1}$ and $e_{2m+1}
= \eta_{m,p}$.

Then $e_i$ generates $J_i = V_i/V_{i+1}$ as a $\GL_2(\Z_p)$-representation.
Moreover, in the isomorphism $J_i = \sigma_{a_i}(b_i) = E[x,y]_{a_i}$, 
the image of $e_i$ corresponds to $x^{a_i}$ (up to some scalar).

\subsubsection{Elements $v_i$}

Let $\Lambda_i = I(V_i)$, and denote by $T_i$ the Hecke operator as in 
Paragraph \ref{compactind} acting
on $\Lambda_i$. We set $v_i = [1,e_i]$. 

Let $w_{i,1} = \sum_{u\in I_1}[g^0_{1,u},e_i]$ if $a_i > 0$, and
$w_{i,1} = \sum_{u\in I_1}[g^0_{1,u},e_i] + [g^1_{0,0},e_i]$ if $a_i =
0$. Then $w_{i,1} = T_iv_i$. 

Let $w_{i,2} = \sum_{u\in I_2}[g^0_{2,u},e_i]$ if $a_i > 0$, and
$w_{i,2} = \sum_{u\in I_2}[g^0_{2,u},e_i] + \sum_{u\in
I_1}[g^1_{1,u},e_i] + [1,e_i]$ if $a_i =
0$. Then $w_{i,2} = T_i^2v_i$. 

\subsection{Properties of the filtration}
\label{propertiesfiltr}

We summarize here some interesting properties of this particular
filtration with respect to the computations of Section
\ref{computefromrel}. These facts come from \cite[Remark 4.4]{BG09}.
They allow us to know without computation the value of $F_i$ for some
steps of the filtration.

\begin{lemm}
\label{Ttheta}
If $r \geq d(p+1)$, and $f \in R[X,Y]_{r-d(p+1)}$,
then for all $g$, $p^{-d}T[g,\theta^df]$ is in $I(\Sym^rR^2)$.
\end{lemm}


\begin{coro}
\label{filtrzero}
Suppose that $d$ is an integer with $v(a_p) < d$ and $r \geq d(p+1)$, 
then for any $f \in E[X,Y]_{r-d(p+1)}$,
and for all $g\in G$, $[g,\theta^df]$ 
is in $M_{k,a_p}$. In
particular, for the standard filtration, $F_i = 0$ for any $i > 2\lfloor a_p \rfloor + 1$.
\end{coro}

\begin{lemm}
The element $[g,X^r]$ is in $M_{k,a_p}$ for any $k$ and $a_p$.
\end{lemm}

\begin{coro}
For the standard filtration, $F_1 = 0$ for any $k$ and $a_p$.
\end{coro}

\section{Elementary divisors and reduction}
\label{computerelgen}

We denote by $R$ any discrete valuation ring, 
with uniformizer $\varpi$, valuation $v$, residue field $E$, and
fraction field $\K$. Let $M$ be a submodule of $R^n$. We denote by $m$ its
rank as a free module over $R$.

Let $\pi$ be the reduction modulo $\varpi$ map from $R$ to $E$, or $R^n$
to $E^n$.

In this Section we state general results about the following problem:
compute the image in $E^n$ of $\pi\left((\cup_d\varpi^{-d}M)\cap
R^n\right)$. We will apply these results in 
Section \ref{computerel} to the usual values of $R$, $\K$,
$E$.

\subsection{Elementary divisors}

From the theory of elementary divisors, we know that there exists a
uniquely determined unordered list $\{d_1,\dots,d_m\}$ of non-negative integers,
such that there exists a basis $\{e_1,\dots,e_n\}$ of $R^n$, such that 
the set $\{\varpi^{d_i}e_i\}_{1\leq i \leq m}$ is an $R$-basis of $M$.
We call the $d_i$s the elementary divisors of $M$ in $R^n$.
We denote by $\delta_M$ the maximum of the $d_i$.

It is clear that:
\begin{prop}
\label{denomelemdiv}
In $\K^n$ we have:
$\cup_d(\varpi^{-d}M \cap R^n) = \varpi^{-\delta_M}M\cap R^n$, and
$\delta_M$ is the smallest integer with this property.
\end{prop}

For later use we reformulate the statement as:

\begin{coro}
\label{denomelemdivmap}
Suppose that $M$ is the image of an injective linear map $t : R^m \to
R^n$, which we extend to a map $\K^m \to \K^n$ by linearity.
Then $\delta_M$ is the smallest of the integers $d$ satisfying:
for all $w \in \K^m$, if $t(w) \in R^n$
then $w$ is in $\varpi^{-d}R^m$.
\end{coro}

\subsection{Good bases of $M$}

Let $(e_i)_{1\leq i \leq n}$ a basis of $R^n$, we denote by $\lambda_i$
the $i$-th coordinate function, so that $v = \sum_i\lambda_i(v)e_i$ for
all $v \in R^n$.

If $x \in R^n$ is a non-zero vector, we define its valuation $v(x)$ to be
$\inf_{1\leq i \leq n}v(\lambda_i(x))$, and its normalization to be
$\tilde{x} = \varpi^{-v(x)}x \in R^n$.

\begin{defi}
\label{goodbasis}
Let $(w_1,\dots,w_m)$ be a
family of vectors of $M$.
We say that it is a good basis of $M$
if: after reordering if necessary the elements of the basis $(e_i)$, we have that
$\lambda_{i}(w_i) \neq 0$, $v(\lambda_{i}(w_i)) = v(w_i)$, 
and for all $j < i$, $\lambda_{j}(w_i) = 0$.
\end{defi}

\begin{prop}
\label{goodbasiselemdiv}
If $(w_1,\dots,w_m)$ is a good basis of $M$, then it is a basis of $M$
and the elementary divisors of $M$ are the $v(w_i)$ for $1 \leq
i \leq m$.
\end{prop} 

\begin{proof}
Consider the matrix of the coordinates of the $w_i$ in the basis of the
$(e_i)$ (reordered if necessary). 
After elementary
operations on the columns we get a diagonal matrix with diagonal elements
the $\lambda_{i}(w_i)$.
\end{proof}

\subsection{Finding a good basis}
\label{Gauss}

In order to find a good basis of $M$, we can use a variant of the Gauss
algorithm. 
We start with a generating family $\V$ for $M$, 
an empty list of indices $\I$, and an empy list of vectors $\V_0$.

\begin{enumerate}
\item 
Find an element $w\in \V$, and an index $i \in I$, such that we have
simultaneously $v(\lambda_i(w)) = v(w)$
and $v(\lambda_i(w)) = \inf_{w'\in \V}v(\lambda_i(w'))$. If we can not do
this, it means that the only vectors left in $\V$ are zero, and we stop. 

\item 
Use $(w,i)$ as a pivot, that is, for all other elements $w'\in \V$, replace
$w'$ by $w' -\left(\lambda_i(w')\lambda_i(w)^{-1}\right)w$.  

\item 
Remove $w$ from $\V$ and put it in an auxiliary list $\V_0$, and
put $i$ at the end of the ordered list $\I$. Then go back to step (1).
\end{enumerate}

At the end of this process, we 
reorder $\{1,\dots,n\}$ by putting first the elements of $\I$
in order, and then all remaining elements of $\{1,\dots,n\}$ in any
order. Then the elements of $\V_0$ form a good basis of $M$.

\begin{rema}
Note that we still have many options as to how we select a vector in step
(1) of the algorithm. It is an interesting question how to do it in the
most efficient way.
\end{rema}

\subsection{Good bases up to $\varpi^d$}

\begin{defi}
\label{goodbasistod}
Let $d > 0$ be an integer. A family $(w_1,\dots,w_s)$ of vectors of $M$
is a good basis up to $\varpi^d$ if (after reordering if necessary the elements
$(e_i)$) we have that $\lambda_{i}(w_i) \neq 0$, $v(\lambda_{i}(w_i)) = 
v(w_i) \leq d$, and for all $j < i$,
$v(\lambda_{j}(w_i)) > d$, and $M$ is generated by  $\{w_i, {1\leq
i\leq s}\} \cup \left(\varpi^{d+1}R^n \cap M\right)$.
\end{defi}

\begin{prop}
\label{gooddtogood}
Let $(w_1,\dots,w_s)$ be a good basis up to $\varpi^d$ of $M$. Then there exists
a good basis $(w_1',\dots,w_m')$ of $M$ such that for all $i \leq s$ we
have $v(w_i) = v(w_i')$ and $\pi(\tilde{w_i}) = \pi(\tilde{w_i}')$, and
for all $i > s$ we have $v(w_i') > d$.
\end{prop}

\begin{proof}
First we replace each $w_i$ for $i > 1$ by $w_i' = w_i +
\left(\lambda_1(w_i)\lambda_1(w_1)^{-1}\right)w_1$. Then for all $i > 1$, we
have $\lambda_1(w_i') = 0$, and moreover the coefficients of $w_i'$ differ
from the coefficients of $w_i$ by elements of valuation at least
$d+1$. In particular $v(w_i') = v(w_i) = v(\lambda_i(w_i))$ and 
$\pi(\tilde{w_i}) = \pi(\tilde{w_i}')$. We iterate this construction
until we get a family of vectors $(w_i')_{1\leq i \leq s}$ that are a
good basis up to $\varpi^d$
satisfying moreover that for all $j < i$, $\lambda_{j}(w_i') = 0$, so
they are part of a good basis of $M$. Then we take a family of vectors in
$\varpi^{d+1}R^n \cap M$ such that $M$ is generated by the $(w_i')$ and this
family, and we apply the Gauss algorithm to it in order to get a good
basis of $M$.
\end{proof}

\begin{prop}
\label{computemod}
Let $d >0$ be an integer. Starting from a generating family for $M$, we
can obtain the image of $(R/\varpi^{d+1}R)^n$ of a good basis up to
$\varpi^d$ by working only in $(R/\varpi^{d+1}R)^n$.
\end{prop}

\begin{proof}
Start with a generating family for $M$, and consider its image in
$(R/\varpi^{d+1}R)^n$. 
We want to apply the same algorithm as described in Paragraph \ref{Gauss}, which needs some
adaptations.
We define the valuation of a non-zero element of $R/\varpi^{d+1}R$ as the
valuation of any of its lifts to $R$. If $a, b$ are two elements of
$R/\varpi^{d+1}R$ with $v(a) \geq v(b)$, we denote by $ab^{-1}$ a choice
of an element $u$ of $R/\varpi^dR$ satisfying $a = ub$ (equivalently,
$ab^{-1}$ is the
reduction modulo $\varpi^{d+1}$ of $\hat{a}\hat{b}^{-1}$ for a choice of
lifts $\hat{a},\hat{b}$ of $a,b$). With these definitions, we can apply
the algorithm. In step (2), note that we had to make some choice in the
definition of $\lambda_i(w')\lambda_i(w)^{-1}$, and the vector
$\left(\lambda_i(w')\lambda_i(w)^{-1}\right)w'$ is not independent of
this choice if $v(w') < v(w)$,
but this vector still has its $i$-th coordinate equal to $0$ in
$R/\varpi^{d+1}R$ independently of the choice.

At the end of this algorithm, we get (after a reordering of the
coordinates if necessary) a family $(w_j)$ of vectors in $(R/\varpi^{d+1}R)^n$ satisfying 
the properties: there exists $s$ such that for $j \leq s$,
$\lambda_{j}(w_j) \neq 0$, $v(\lambda_{i}(w_j)) = v(w_j)$, and
for all $i < j$, $\lambda_{i}(w_j) = 0$, and for $j > s$, $w_j = 0$.

We need now to show that the set of the first $s$ vectors of this family is the image in 
$(R/\varpi^{d+1}R)^n$ of a good basis up to $\varpi^d$ of $M$. For this, we do
to the family of vectors of $M$ the same operations as we did modulo
$\varpi^{d+1}$, choosing in step (2) any choice of lift of the
coefficient $\lambda_i(w')\lambda_i(w)^{-1}$. Then when we stop the algorithm
for the vectors modulo $\varpi^{d+1}$, what we get for the original
vectors in $R^n$ is a family of $s$ vectors that is a good basis of $M$
up to $\varpi^d$, and the rest is a family of vectors with coefficients in
$\varpi^{d+1}R^n$.
\end{proof}

\begin{coro}
\label{isreallyabasis}
Suppose that we start from a generating family that is in fact a basis of
$M$, and that in the process we do not obtain the zero vector of
$(R/\varpi^{d+1}R)^n$. Then the family we obtain at the end is the
reduction modulo $\varpi^{d+1}$ of a good basis up to $\varpi^d$ of $M$
which is also a basis of $M$.
\end{coro}

\begin{coro}
\label{dsmallerdeltaM}
Suppose that we start from a basis of $M$. Then in the process we
do not obtain the zero vector of $(R/\varpi^{d+1}R)^n$ if and only if
$\delta_M \leq d$.
\end{coro}

\begin{proof}
Both properties are equivalent to the fact that, for the basis given by
Proposition \ref{gooddtogood}, we have $s = m$.
\end{proof}

\subsection{Computing some reductions modulo $\varpi$}

We consider all modules here as submodules of $\K^n$.

\begin{prop}
\label{reducdenom}
Let $d > 0$ be an integer and let $(w_i)_{1\leq i \leq s}$ be a good
basis of $M$ up to $\varpi^d$.
Then $\pi(\varpi^{-d}M \cap R^n)$ is generated by the $\pi(\tilde{w_i})$,
$1 \leq i \leq s$.

Suppose that moreover $d \geq \delta_M$. Then $\pi(\cup_t(\varpi^{-t}M \cap R^n))$ is
generated by the $\pi(\tilde{w_i})$.
\end{prop}

\begin{proof}
Let $w_i$ be an element of the basis such that $v(w_i) \leq d$.  Then
there exists $\mu\in R$ with $v(\mu) = d-v(w_i)$, and then
$\varpi^{-d}\mu w_i$ is in $\varpi^{-d}M \cap R^n$ and reduces in $E$ to a non-zero
multiple of $\pi(\tilde{w_i})$. So all the $\pi(\tilde{w_i})$ with
$v(w_i) \leq d$ are in $\pi((\varpi^{-d}R^m) \cap R^n)$.

Let now $(w_i')_{1\leq i \leq m}$ be a good basis of $M$ as in
Proposition \ref{gooddtogood}.
Let $x \in \varpi^{-d}M$. We can write $x = \sum
\varpi^{-d}\mu_i w_i'$ for some elements $\mu_i\in R$. 
Then $x = \sum \mu_i\varpi^{v(w_i')-d}\tilde{w}_i'$. By looking at the coefficients of
$x$ in the basis $(e_i)$ we see that $x\in R^n$ if and only if 
$\varpi^{v(w_i')-d}\mu_i \in R$ for all $i$. 
Then $\pi(x) = \sum_i \pi(\varpi^{v(w_i')-d}\mu_i)\pi(\tilde{w_i}')$, and
this is the same as taking the sum over only the indices $i$ for which
$v(w_i') \leq d$, that is $i \leq s$. So $\pi(x)$ is in the subspace of
$E^n$ generated by the $\pi(\tilde{w_i}') = \pi(\tilde{w_i})$ for $i
\leq s$.  

The last statement follows from Proposition \ref{denomelemdiv}.
\end{proof}

\begin{coro}
\label{finitering}
Let $d$ be an integer. Starting from a generating family for $M$, we can
compute a generating family for $\pi(\varpi^{-d}M \cap R^n)$ by doing
computations in $(R/\varpi^{d+1}R)^n$.

Let $d$ be an integer larger than or equal to $\delta_M$.
Then we can compute a generating family for $\pi(\cup_t(\varpi^{-t}M \cap
R^n))$ by doing computations in $(R/\varpi^{d+1}R)^n$.
\end{coro}

\subsection{Computing locally}
\label{computelocal}

We describe here a way to choose the vectors we work with in the Gauss
algorithm which allows to work first with some subset of all the
coordinates, ignoring all other coordinates. We will apply this in
Section \ref{explicit} to our original problem.

Let $\V$ be a family of vectors that generates $M$, $I$ a subset of
$\{1,\dots,n\}$, $\V_I \subset \V$ the subset of vectors with
support in $I$. Suppose that we have $J \subset I$ such that for all $w
\in \V$, $w$ not in $\V_I$, the intersection of the support of $w$ with
$I$ is contained in $J$. Then we can compute a good basis of $M$ (or the
reduction modulo $\varpi^{d+1}$ of a good basis of $M$ up to $\varpi^d$) in two
steps as follows:
\begin{enumerate}
\item
use a partial Gauss algorithm:
apply the Gauss algorithm for vectors in $\V_I$, but allowing only the
use of pivots that are in $I \setminus J$. 
\item
when we can not do anything anymore, we have obtained a family $\V_{I,0}$
of vectors that were used at pivots, and a remainder $\V_I'$.
\item
apply the Gauss algorithm to the family $(\V \setminus \V_I) \cup \V_I'$.
\item
when the only vectors left in the family are zero vectors, we have
obtained a family $\V_0'$ of vectors that were used at pivots. 
\end{enumerate}

Then $\V_{I,0} \cup \V_0'$ is a list of vectors obtained using the Gauss
algorithm on the whole family $\V$. We call $\V_{I,0}$ the set of
extracted vectors for this partial Gauss algorithm, and the coordinates
in $J$ are the excluded coordinates.

\section{Computing relations}
\label{computerel}

In this section, we explain how to apply the general results of Section
\ref{computerelgen} to obtain finite-dimensional subspaces of
$M_{k,a_p}$, in order to apply the results of Section
\ref{computefromrel}.

\subsection{Description of the operator $T$}
\label{descrT}

We give an explicit description of the $G$-equivariant operator $T$ on
$I(\Sym^r \K^2)$ or $I(\Sym^r E^2)$ introduced in Paragraph
\ref{compactind}.

Let $f = [g,v]$, with $v$ a polynomial with coefficients in $\K$ or $E$. 
Then it follows easily from Lemmas 2.2.1 and 2.2.2 of \cite{Br03b} that
we have the following formula:
$$Tf = \sum_{u\in I_1}[gg^0_{1,u},v(X,pY-uX)]+[gg^1_{0,0},v(pX,Y)]$$

For $f = [g,v]$, we denote by $T^+f$ the part of $Tf$ with support farther from the origin
than $g$, and $T^-f$ the part with support closer to the origin.
So if $g = g^0_{n,\lambda}$, we get that
$T^+f = \sum_{u\in I_1}[gg^0_{1,u},v(X,pY-uX)]$
and $T^-f = [gg^1_{0,0},v(pX,Y)]$.
When $g = 1$ by convention we set 
$T^+f = \sum_{u\in I_1}[g^0_{1,u},v(X,pY-uX)]$ and
$T^-f = [g^1_{0,0},v(pX,Y)]$.
We extend the definition of $T^+$ and $T^-$ by linearity to $I(\Sym^r)$.

We introduce the notation $(g^0_{n,\mu},i) = [g^0_{n,\mu},X^{r-i}Y^i]$,
and $(g^1_{n,\mu},i) = [g^1_{n,\mu},X^iY^{r-i}]$.
We have $\beta g^0_{n,\lambda} = g^1_{n,\lambda} w$,
where $w = \matr 0110$ and $\beta = \matr 01p0$, so
$\beta(g^{\eps}_{n,\mu},i) = (g^{1-\eps}_{n,\mu},i)$ for $\eps \in
\{0,1\}$.

We also have $wg^0_{n,p\mu} = g^1_{n-1,\mu}w$ for all $\mu\in I_{n-1}$,
and so $w(g^0_{n,p\mu},i) = (g^1_{n-1,\mu},i)$.

Using the formula above, we can
give an explicit formula for $Tf$ for $n > 0$:
$$
T^+(g^{\eps}_{n,\lambda},i) = 
\sum_{u\in I_1}
\sum_{j=0}^ip^j\binom{i}{j}(-u)^{i-j}
(g^{\eps}_{n+1,\lambda+up^n},
j)
$$
and
$$
T^-(g^{\eps}_{n,\lambda},i) = 
p^{r-i}\sum_{j=0}^i\binom{i}{j}(\lambda')^{i-j}
(g_{n-1,\underline{\lambda}}^{\eps},j)
$$
where we set $\lambda = \underline{\lambda} + p^{n-1}\lambda'$, with
$\lambda' \in I_1$ and $\underline{\lambda} \in I_{n-1}$.

We also have:
$$
T(g^{\eps}_{0,0},i) = 
\sum_{u\in I_1}
\sum_{j=0}^ip^j\binom{i}{j}(-u)^{i-j}
(g_{1,u}^\eps,j)
+
p^{r-i}(g_{0,0}^{1-\eps},r-i)
$$

\begin{prop}
\label{Tisinjective}
For any $a \in \K$, the operator $(T-a)$ is injective on $I(\Sym^r \K^2)$.
\end{prop}

\begin{proof}
Let $f \in I(\Sym^r \K^2)$, $f \neq 0$, and take the smallest $n$ such
that $f$ has support in $B_n$. Then it is clear that the part of
$(T-a_p)f$ that has support in $C_{n+1}$ is non-zero.
\end{proof}

We say that $f \in I(\Sym^r\K^2)$ is integral if $f \in I(\Sym^r R^2)$.
It is clear that if $f$ is integral, then so is $Tf$.

\subsection{Denominators}

Let $\M = \M_{k,a_p}$ be
the submodule $((T-a_p)I(\Sym^r \K^2) \cap I(\Sym^r R^2))$. 

We recall the following result from \cite[Proposition 3.3.3]{Br03b}:

\begin{prop}
\label{denominatorsmallr}
Suppose $r \leq p-1$.
Let $f \in I(\Sym^r \K^2)$ such that $(T-a_p)f$ is
integral. Then $f$ is integral.
\end{prop}

We give a generalization of this result for larger values of $r$.
Let $N_r = 0$ if $r=0$, and $N_r = \lfloor (r-1)/(p-1) \rfloor$ if $r
\geq 1$ (so in particular $N_r = 0$ for all $r < p$).

\begin{prop}
\label{denominatorn}
Assume $r \geq p$.
Let $f \in I(\Sym^r \K^2)$ with support in $B_n$ such that $(T-a_p)f$ is
integral. Write $f = \sum_{i=0}^n f_i$ with $f_i$ having
support in $C_i$.
Then $p^{(n+1-i)N_r}f_i$ is integral.
\end{prop}

\begin{coro}
\label{denominatorngen}
Let $f \in I(\Sym^r \K^2)$ with support in $B_n$ such that $(T-a_p)f$ is
integral. 
Then $p^{(n+1)N_r}f$ is integral.
\end{coro}

In order to prove Proposition \ref{denominatorn},
we start by a special case:

\begin{lemm}
\label{denominator1}
Assume $r \geq p$.
Let $f = (g^\eps_{n,\mu},v) \in I(\Sym^r \K^2)$ with $T^+f$ integral. Write 
$f = \sum_{i=0}^r(-1)^i c_i (g^\eps_{n,\mu},i)$.
Then $c_i \in p^{-\min(i,N_r)}R$.
\end{lemm}

\begin{proof}
We set $c_i = 0$ if $i < 0$ or
$i > r$, 
and $\alpha_{j,u} = p^j\sum_{i\geq 0}c_i\binom{i}{j}u^{i-j} \in R$.

Then $T^+f
=\sum_{j=0}^r\sum_{u\in I_1}p^j(-1)^j\alpha_{j,u}(g^\eps_{n+1,\mu+p^nu},j)$. 
So this is integral if and only
if for all $j$ and for all $u\in I_1$, 
$\alpha_{j,u} \in R$

Looking at $\alpha_{0,0}$ we see that $c_0 \in R$, and more generally
looking at $\alpha_{i,0}$ we see that $c_i \in p^{-i}R$.
For $a \in \Z/(p-1)\Z$, let 
$\beta_{j,a} = (\sum_{u\in \mu_{p-1}} u^{-a}\alpha_{j,u})/(p-1)$,
except for 
$\beta_{0,0} = (\sum_{u\in \mu_{p-1}} u^{-a}\alpha_{0,u})/(p-1) - c_0$.

Then
$$\beta_{j,a} = p^j \sum_{\stackrel{i > 0}{i = a + j \pmod{p-1}}}c_i\binom{i}{j}
$$
and for all $j,a$, we get that $\beta_{j,a}\in R$.

We fix now $b\in \{1,\dots,p-1\}$, and consider only the elements 
$c_{b+\ell(p-1)}$, $\ell
\geq 0$. Let $n$ be the largest integer such that $b + \ell(p-1) \leq r$,
so we have $(n+1)$ unknowns $x_0,\dots,x_n$ with $x_i = c_{b+i(p-1)}$. 
We consider the $(n+1)$ equations for $0 \leq j \leq n$:
$$
\sum_{m = 0}^n \binom{b+m(p-1)}{j}x_m = p^{-j}\beta_{j,b-j}
$$
Using Lemma \ref{determinant},
we get that $x_m \in p^{-n}R$ for all $m$.

Now we compute the value of $n$.
Write $r = N_r(p-1) + a$ with $1 \leq a \leq p-1$.
Then $n = N_r$ for $b \leq a$, and $n = N_r-1$ for $b > a$.
So in any case $n \leq N_r$.
\end{proof}

\begin{lemm}
\label{determinant}
For all $b \in \Z$ and $n \geq 0$, the determinant $\delta_{b,n}$ of the
matrix with coefficients
$(\binom{b+m(p-1)}{j})_{0\leq j,m \leq n}$ is
invertible in $R$.
\end{lemm}

\begin{proof}
Let $z_{m,j} = \prod_{i = 0}^{j-1}(b+m(p-1)-i)$.
Then $\delta_{b,n} = (\prod_{i= 0}^n (i!))^{-1} \det((z_{m,j})_{0\leq j,m\leq
n})$. Note that $z_{m,j} = P_j(b+m(p-1))$ for a polynomial $P_j$ that is
monic of degree $j$ and independent of $b$ and $m$. So
by linearity $\det((z_{m,j})_{0\leq j,m\leq n})
= \det( (z'_{m,j})_{0\leq j,m \leq n})$ where $z'_{m,j} = (b+m(p-1))^j$.
But the latter is the Vandermonde determinant on the $b+m(p-1)$, $0 \leq
m \leq n$, so is equal to $\prod_{0 \leq i < j \leq n}(b+j(p-1) - b
-i(p-1)) = (p-1)^{n(n+1)/2}\prod_{i= 0}^n (i!)$.
So finally $\delta_{b,n} = (p-1)^{n(n+1)/2}$.
\end{proof}

\begin{proof}[Proof of Proposition \ref{denominatorn}]
Write $f = \sum_{i = 0}^n f_i$ with $f_i$ having support in $C_i$. 

We see first from Lemma \ref{denominator1} that $p^{N_r}f_n$ is integral.
Indeed, $T^+f_n$ is integral, as it is exactly the part of $(T-a_p)f$
with support in $C_{n+1}$. 

The part of $(T-a_p)f$ with support in $C_n$ is $-a_pf_n + T^+f_{n-1}$.
As $(p^{N_r})(-a_pf_n)$ is integral, we see that $p^{N_r}T^+f_{n-1}$ is integral,
and then so is $p^{2N_r}f_{n-1}$.

Let $i < n-1$, and consider the part of $(T-a_p)f$ with support in $C_i$.
It is $T^-f_{i+1} -a_pf_i + T^+f_{i-1}$. If $p^{(n+1-i)N_r}(T^-f_{i+1}
-a_pf_i)$ is integral, then so is $p^{(n+1-i)N_r}(T^+f_{i-1})$, and so is
$p^{(n+2-i)N_r}f_{i-1}$ by Lemma \ref{denominator1}.
\end{proof}

\subsection{Finite dimensional subspaces of $M_{k,a_p}$}

Recall that $M_{k,a_p}$ is the reduction modulo $p$ of $\M_{k,a_p}$.

Let $S$ be a finite subset of the vertices of the tree of $\GL_2(\Q_p)$.
For any representation $V$ of $K$, we denote by $I_S(V)$ the subspace of
$I(V)$ of vectors whose support in the tree is included in $S$.
For example we can take for $S$ the set $B_n$ of vertices at distance at most
$n$ from the origin, in this case we write $I_n(V)$ for $I_{B_n}(V)$.

For any non-negative integer $d$ and $S$ as above, let $\M_{d,S} 
= (T-a_p)p^{-d}I_S(\Sym^r R^2) \cap I(\Sym^r R^2)$. It is clear that
$\M_{k,a_p}
= \cup_{d,S} \M_{d,S}$, and so $M_{k,a_p} = \cup_{d,S} M_{d,S}$ where $M_{d,S}$
is the image of $\M_{d,S}$ in $M_{k,a_p}$. We also write $\M_S = \cup_d
\M_{d,S}$.

From Proposition \ref{denominatorngen} , we deduce that:

\begin{prop}
\label{elemdivBn}
If $S \subset B_n$,
the maximum of the elementary divisors of $(T-a_p)I_{S}(\Sym^r R^2)$ inside
$I_{n+1}(\Sym^r R^2)$ is at most $(n+1)N_r$.
\end{prop}

\begin{rema}
From the computations, it seems that this bound is far from optimal. In
the notation of the tables of Section \ref{examples}, the bound given in
the Proposition states that $\delta \leq (n+1)\lfloor (k-3)/(p-1) \rfloor$. 
But the examples computed always give $\delta \leq 2(n+1)$. We expect
$\delta/(n+1)$ to grow with $k$, but probably more slowly than linearly. 
\end{rema}

And so, using Corollary \ref{denomelemdivmap}:

\begin{prop}
\label{denomS}
Suppose that $S \subset B_n$. Then 
$$
\cup_{d \geq 0}\left(\varpi^{-d}(T-a_p)I_{S}(\Sym^r R^2) \cap I(\Sym^r
R^2)\right) = p^{-(n+1)N_r}
(T-a_p)I_{S}(\Sym^r R^2) \cap I(\Sym^r R^2)
$$
\end{prop}

\begin{coro}
\label{boundeddenominator}
Assume that $S \subset B_n$, and that we can compute
$\bar\Theta_{k,a_p}^{ss}$ by using only elements of $\M_{S}$.
Then we can compute $\bar\Theta_{k,a_p}^{ss}$ by using only elements
of $\M_{\delta,S}$ for $\delta = (n+1)N_r$.
\end{coro}

In particular, as explained in Section \ref{computerelgen}, this means
that we can work with $R/p^{\delta+1}R$-modules instead of $R$-modules,
and so work with finite precision, and Corollary \ref{boundeddenominator} 
gives us an explicit
bound on the precision needed in terms of $S$.
We summarize this informally as:

\begin{coro}
\label{finiteprec}
For a fixed $S$, we need only work with finite precision determined
explicitly in terms of $S$.
\end{coro}

\subsection{An explicit algorithm for the computation}
\label{explicit}

Consider now the following problem: we fix some $n \geq 0$, and we want
to compute a good basis (or a good basis up to $\varpi^d$ for some $d$) for the module
$\M_n = (T-a_p)I_n(\Sym^r R^2)$.
The elements of $\M_n$ have support in $B_{n+1}$, so this means doing the Gauss
algorithm as in Paragraph \ref{Gauss} with a family of size $(r+1)\# B_n$
in a module of rank $(r+1)\# B_{n+1}$.  We now explain how to take
into account the structure of the tree to speed up significantly the
computation. 

We take as basis vectors of the ambient space the elements
$[g^\eps_{m,\mu},X^{r-i}Y^i]$ for $m+\eps \leq n+1$, $\mu\in I_m$, $0\leq
i \leq r$. Denote by $\B_m$ the
set of elements of the form 
$[g^\eps_{m-\eps,\mu},X^{r-i}Y^i]$ for $0 \leq i \leq r$, so $\B_m$ is a
basis of $I_{C_m}(R)$.

The module $\M_n$ is generated by the $(T-a_p)f$ for $f \in
\cup_{i=0}^n\B_i$, and these elements form in fact a basis of $\M_n$ by
Proposition \ref{Tisinjective}.

\subsubsection{Subdivising the tree}

For $m \leq n$, let $S_m$ the subtree of $B_{n+1}$ containing the element
$[g^0_{n+1-m,0}]$ and all classes $[g]$ of $B_{n+1}$ that are farther than
$[g^0_{n+1-m,0}]$ from $[g^0_{0,0}]$ (that is, the path from $[g^0_{0,0}]$ to
$[g]$ goes through $[g^0_{n+1-m,0}]$).

We denote by $\B_i^m$ the subset of elements $f$ of $\B_i$ such that 
$(T-a_p)f$ has support in $S_m$. 
This is empty if $i \leq n+1-m$, and if $i > n+1-m$ this is the same as
the set of $f \in \B_i$ with support in $S_m$.

If $x \in \Z_p$, let $t(x)$ be the matrix $\smatr 1x01$.
Then note that $S_{m+1}$ is the disjoint union of $\{[g^0_{n-m,0}]\}$ and the
$t(up^{n-m})S_m$ for $u \in I_1$.

Moreover, $B_{n+1}$ is the disjoint union of $\{[g^0_{0,0}]\}$, the 
$t(u)S_n$ for $u \in I_1$, and $w S_n$.

\subsubsection{Computing locally}

We now make use the remark of Paragraph \ref{computelocal}, and obtain
the following algorithm:

\begin{itemize}
\item
Do the partial Gauss algorithm for the family of vectors $(T-a_p)f$ for $f
\in \B^2_n$,
excluding the coordinates $[g^0_{s,\mu},X^{r-i}Y^i]$
for $s = n-1$ or $s = n$. We get a list $\V_{2,0}$ of extracted vectors
and a list of remaining vectors $\V_2'$.

\item
For $m$ from $3$ to $n$: do the partial Gauss algorithm for vectors in
the list 
$$\cup_{u\in I_1}t(up^{n-m+1})\V_{m-1}' \cup \{(T-a_p)f, f \in \B^m_{n+2-m}\}$$ 
excluding the coordinates $[g^0_{s,\mu},X^{r-i}Y^i]$
for $s = n+1-m$ or $s = n+2-m$. 
We get a list $\V_{m,0}$ of extracted vectors, and a list of
remaining vectors $\V_m'$.

\item
Do the Gauss algorithm for the list 
$$\cup_{u\in I_1}t(u)\V_n' \cup
w\V_n'\cup \{(T-a_p)f, f \in \B_1 \cup \B_0\}$$ 
We get a list $\V_{n+1,0}$ of extracted vectors.
\end{itemize}

In the end, we get a list $\V_0$ of extracted vectors which is equal to 
$\V_{n+1,0} \cup (\cup_{m = 2}^n\cup_{\mu\in I_{n+1-m}} t(\mu)\V_{m,0})
\cup w(\cup_{m = 2}^n\cup_{\mu\in I_{n+1-m}} t(\mu)\V_{m,0}) $.

\begin{rema}
The advantage of this method is that we do most of our computations in
modules that are of rank much smaller than the rank of
$I_{n+1}(\Sym^r R^2)$,
as $I_{S_m}(\Sym^r R^2)$ is of rank $(r+1)(p^{m+1}-1)/(p-1)$. The final
computation in $I_{n+1}(\Sym^r R^2)$ is still the longest part, but at this stage the
number of vectors that we have to take into account is much smaller that
the rank of $I_n(\Sym^r R^2)$, which is the number we would have to consider
otherwise.
\end{rema}

\section{Summary of the algorithm}

\subsection{The algorithm}
\label{summary}

We summarize the algorithm to compute $\bar{V}_{k,a_p}$. Let $\K =
\Q_p(a_p)$ with ring of integers $R$ and uniformizer $\varpi$. Let $e$ denote
the ramification index of $\K$, so that computing modulo $p^d$ in $R$ is
the same as computing modulo $\varpi^{ed}$.

\begin{enumerate}
\item
For some $n \geq 0$ and $d \geq 0$,
compute the reduction modulo $p^{d+1}$ of a good basis up to 
$p^d$ for the submodule $(T-a_p)I_n(\Sym^r R^2)$ of $I_{n+1}(\Sym^r R^2)$, 
by the method of Paragraph \ref{explicit}.

\item
Deduce from this a generating family for 
the subspace $M_{d,B_n}$ of $M_{k,a_p}$, via
Proposition \ref{reducdenom}.

\item
Using the elements given in Paragraph \ref{explicitbasis} and the method
of Paragraph \ref{explicitF}, get
information on the graded parts $F_i$ of the filtration of
$I(\sigma_{r,E})/M_{k,a_p}$ described in Section \ref{computefromrel}.

\item
If we have enough information on the $F_i$ to deduce
$\JH(\bar\Theta_{k,a_p})$, stop here. 
If we do not have enough information,
try again with a larger value of $n$ or of $d$.
\end{enumerate}

The algorithm will give enough information to deduce 
$\JH(\bar\Theta_{k,a_p})$ for some finite value of $n$.

\begin{rema}
By Corollary \ref{dsmallerdeltaM}, it is easy to check during step (1) of
the algorithm if trying the computation with the same $n$ but a larger
value of $d$ would add any new information.
\end{rema}

\begin{rema}
As a byproduct of the computation, we can take note of the largest of the
elementary divisors that appeared and deduce a local constancy result via
Corollary \ref{locconstapbis} and Corollary \ref{locconstrbis}.
\end{rema}

\subsection{Remarks on the choice of $d$ and $n$}

\subsubsection{Choice of $d$}

The first problem is how to best choose the value of $d$ for a given $n$.
We would like to take $d$ large enough that we get reduction of all
the elements in $\cup_{t \geq 0}\left(\varpi^{-t}(T-a_p)I_{B_n}(\Sym^r R^2) \cap
I(\Sym^r R^2)\right)$ (so that we do not have to try again with the same
value of $n$ but a larger $d$).  By Proposition \ref{elemdivBn}, we could
take $d = (n+1)N_r$ to ensure that this property is verified.  However,
this choice is probably not optimal: indeed we do not want to take $d$
too large as the computation time would be much longer. In fact, from the
examples, the largest elementary divisor that appears 
for $(T-a_p)I_{B_n}(\Sym^r R^2)$ inside $I(\Sym^r R^2)$
seems to be much
smaller than $(n+1)N_r$.  So it is probably better to choose $d$
based on an estimation of the size of the largest elementary divisor, for
example using the value obtained from a computation with a smaller value
of $n$.  

\subsubsection{Choice of $n$}

From a theoretical point of view, it would be very interesting to have a
bound for the value of $n$ necessary for this computation.
However, from a computational point of
view, having such a bound is quite useless unless we know in fact exactly
the smallest necessary value of $n$. Indeed, let $d_n$ be the rank
of the $R$-module $I_n(\Sym^r R^2)$. Then, for fixed $r$, $d_n$ grows as
$p^n$. The computation time grows at least as fast as $d_n^{\alpha}$ for
some $\alpha \geq 1$, so it grows at least as fast as $x^n$ for some $x >
2$. So in particular, if $n_0$ is the smallest value that allows us to
compute the reduction, then doing the computation for some $n > n_0$ is
expected to take actually more time than doing the computation
successively for $n = 1,2 \dots n_0$.

\section{Local constancy results}
\label{locconst}

\subsection{A local constancy result with respect to $a_p$}
\label{locconstapsect}

We know that the value of $\bar{V}^{ss}_{k,a_p}$ is locally constant with
respect to $a_p$, for a fixed $k$, and we have explicit radii:
see \cite{Ber12} for the case $a_p \neq 0$, and \cite{BLZ} for
$a_p = 0$. We show that the results of the computation of the algorithm
also give an explicit radius around a given $a_p$ for a fixed $k$.

We say that we are able to compute $\bar\Theta_{k,a_p}^{ss}$ by using
elements of $M' \subset M_{k,a_p}$ if there is a unique $\Sigma \in \jh$ such
that $\Sigma \subset \JH(I(\sigma_{r,E})/E[G]M')$. In this case
we have that $\Sigma =
\JH(\bar\Theta_{k,a_p}) = \JH(\bar\Pi_{k,a_p})$.

\begin{theo}
\label{locconstap}
Suppose that we are able to compute $\bar\Theta_{k,a_p}^{ss}$ by using only
elements of $\M_{d,S}$ for some $d$ and $S$ as above. Then for all $a$
such that $v(a -a_p) > d$, we have that $\bar{V}_{k,a_p}^{ss} \simeq
\bar{V}_{k,a}^{ss}$, and $\JH(\bar\Pi_{k,a}) = \Sigma$ for the unique
$\Sigma \in \jh$ such that $\Sigma \subset
\JH(I(\sigma_{r,E})/E[G]M_{d,S}(a_p))$.
\end{theo}

\begin{proof}
Let $\Sigma$ be the unique element of $\jh$ such that
$\Sigma \subset \JH(I(\sigma_{r,E})/E[G]M_{d,S}(a_p))$.
If $v(a -a_p) > d$ then $M_{d,S}(a) = M_{d,S}(a_p)$, so $\Sigma$ is also
the unique element of $\jh$ such that 
$\Sigma \subset \JH(I(\sigma_{r,E})/E[G]M_{d,S}(a))$.
Then $\Sigma = \JH(\bar\Theta_{k,a})$ so 
$\bar\Theta_{k,a}^{ss}$ is constant on the set $\{v(a_p-a) > d \}$. 
So finally $\bar\Pi_{k,a}$ is constant on the set 
$\{v(a_p-a) > d\}$, and its list of Jordan-Hoelder factors is given by
$\Sigma$.
\end{proof}

\begin{coro}
\label{locconstapbis}
Suppose that we are able to compute $\bar\Theta_{k,a_p}^{ss}$ by using
only elements of $\M_{d,S}$ for some $d$ and $S$. Let $\delta_S$ be the maximal
valuation of the elementary divisors of $(T-a_p)I_S(\Sym^r R^2) \subset I(\Sym^r R^2)$. 
Then for all $a$ such
that $v(a - a_p) > \delta_S$, we have that 
$\bar{V}_{k,a_p}^{ss} \simeq \bar{V}_{k,a}^{ss}$.
\end{coro}

\begin{proof}
It follows from Proposition \ref{denomelemdiv} that we can in fact compute
$\bar\Theta_{k,a_p}^{ss}$ by using only elements of $\M_{\delta_S,S}$.
\end{proof}

When we apply the Gauss algorithm as in Paragraph \ref{explicit}, the
value of $\delta_S$ appears automatically as a byproduct. So even if we started
by a non-optimal $d$, we still get a bound for local constancy that is
closer to the optimal one. In particular, as we see in the examples of
Section \ref{examples}, we often obtain a better bound than the one given
by \cite[Theorem A]{Ber12} or \cite[Theorem 1.1.1]{BLZ}.

\subsection{A local constancy result with respect to the weight}
\label{locconstrsect}

\begin{theo}
\label{locconstr}
Let $\delta \geq 0$ be such that we can compute $\bar\Theta_{k,a_p}^{ss}$
using the algorithm of Section \ref{computefromrel} and the standard filtration
by using as denominators only elements of valuation at most $\delta$. Let
$c > v_p(a_p) + \delta$ be an integer.

Suppose that $c \leq (k-2)/(p+1)$. Then for all $k'\geq
k$ such that $k' \equiv k \mod (p-1)p^{1 + \lfloor\delta\rfloor + \lfloor
\log_p(c)\rfloor}$, we have
$\bar{V}_{k',a_p}^{ss} \simeq \bar{V}_{k,a_p}^{ss}$.
\end{theo}

As in Paragraph \ref{locconstapsect}, we deduce from this Theorem the following
Corollary:

\begin{coro}
\label{locconstrbis}
Suppose that we are able to compute $\bar\Theta_{k,a_p}^{ss}$ using the
algorithm of Section \ref{computefromrel} and the standard filtration by
using only elements of $\M_{d,S}$ for some $d$ and $S$. Let $\delta_S$ be
the maximal valuation of the elementary divisors of $(T-a_p)I_S(\Sym^r
R^2) \subset I(\Sym^r R^2)$. 
Let $c > v_p(a_p) + \delta_S$ be an integer.

Suppose that $c \leq (k-2)/(p+1)$. Then for all $k'\geq
k$ such that $k' \equiv k \mod (p-1)p^{1 + \lfloor\delta_S\rfloor + \lfloor
\log_p(c)\rfloor}$, we have
$\bar{V}_{k',a_p}^{ss} \simeq \bar{V}_{k,a_p}^{ss}$.
\end{coro}

\begin{rema}
As we see in Section \ref{examples}, the condition that 
$c \leq (k-2)/(p+1)$ seems to be often but not always satisfied.
\end{rema}

We need a few preliminaries before we can prove Theorem \ref{locconstr}.

\subsubsection{Combinatorial lemmas}

\begin{lemm}
\label{binomial}
Let $i \in \Z_{\geq 0}$. Then for all $x$, $y \in
\Z_p$, we have $v_p(\binom{x}{i}-\binom{y}{i}) \geq \max(0,v_p(x-y) -
\lfloor\log_p(i)\rfloor)$.
\end{lemm}

\begin{proof}
We write $\binom{x}{i} = \binom{y}{i} +
\sum_{j=1}^i\binom{y}{i-j}\binom{x-y}{j}$. Then $\binom{y}{i-j}\in \Z_p$
whereas $\binom{x-y}{j} = \frac{x-y}{j}\binom{x-y-1}{j-1}$ with
$\binom{x-y-1}{j-1}\in \Z_p$ and $v_p(j) \leq \lfloor \log_p(i) \rfloor$,
so each term of the sum $\sum_{j=1}^i\binom{y}{i-j}\binom{x-y}{j}$ has
valuation as least $v_p(x-y)-\lfloor \log_p(i) \rfloor$.
\end{proof}

\begin{lemm}
\label{sumbinom}
Let $d,s,t,n,n',\ell$ be non-negative integers with
$n' \geq n$ and $s+t \leq n$ and $\ell < n$ and $d < n$.
When $p = 2$ suppose moreover that $n \geq d+\ell$.
Suppose that $n' \equiv n \mod
(p-1)p^{d+\lfloor\log_p(\max(\ell,d,s,t))\rfloor}$.
Then for all $b \in \{0,\dots,p-2\}$, we have 
$$
\sum_{j \equiv b \mod p-1,s\leq j\leq n-t}\binom{j}{\ell}\binom{n}{j} \equiv
\sum_{j \equiv b \mod p-1,s\leq j \leq n'-t}\binom{j}{\ell}\binom{n'}{j}
\mod p^d
$$
\end{lemm}

\begin{proof}
Suppose first that $s = 0$ and $t=0$.
Let $f_n(x) = (1+x)^n$, seen as a function over $\Q_p$.
Let $g_n(x) = {\ell!}^{-1}x^{\ell-b}f_n^{(\ell)}(x)$.
Then $g_n(x) = \binom{n}{\ell}x^{\ell-b}(1+x)^{n-\ell} = \sum_{j \geq
0}\binom{j}{\ell}\binom{n}{j}x^{j-b}$.
So 
$$
\sum_{j \equiv b \mod p-1,j\leq n}\binom{j}{\ell}\binom{n}{j}
= 
\frac{1}{p-1}\sum_{\xi\in\mu_{p-1}}g_n(\xi)
$$
and we have the same formula for $n'$. So it is enough to show that
$g_n(\xi) \equiv g_{n'}(\xi) \mod p^d$ for all $\xi\in \mu_{p-1}$. 

Suppose first $p=2$. Then $\xi=1$, $1+\xi = 2$, so as $n-\ell \geq d$ we have
that $g_n(\xi) \equiv 0 \mod 2^d$ and the same for $g_{n'}(\xi)$.

Suppose now that $p \neq 2$.
The congruence is true for $\xi = -1$ as $\ell < n$, so we can suppose $\xi\in
\mu_{p-1}\setminus \{-1\}$. Write $n = a + (p-1)s$ for some $s \geq 0$
so $n' = a + (p-1)s'$
with $s' \equiv s \mod p^{d + \lfloor\log_p(\max(\ell,d)\rfloor}$.
Then $g_n(x) = \binom{n}{\ell}x^{\ell-b}(1+x)^a(1+x)^{(p-1)s}$ and
$g_{n'}(x) = \binom{n'}{\ell}x^{\ell-b}(1+x)^a(1+x)^{(p-1)s'}$.
Note that $1+\xi$ is in
$\Z_p^\times$ when $\xi \neq -1$, so $(1+\xi)^{p-1} \in 1+p\Z_p$. We write
$(1+\xi)^{p-1} = 1+pz_{\xi}$. Then $(1+\xi)^{(p-1)s} = (1+pz_\xi)^s$, so 
$g_n(\xi) \equiv
\binom{n}{\ell}\xi^{\ell-b}(1+\xi)^a\sum_{i=0}^{d-1}\binom{s}{i}p^iz_\xi^i \mod p^d$
and
$g_{n'}(\xi) \equiv
\binom{n'}{\ell}\xi^{\ell-b}(1+\xi)^a\sum_{i=0}^{d-1}\binom{s'}{i}p^iz_\xi^i \mod
p^d$.
Now we apply Lemma \ref{binomial} to get the result.

\bigskip

Now go back to a general $s$ and $t$. The difference between the sum we want and
the sum we have computed for $s=0$  and $t=0$ is a finite number of binomials of the form
$\binom{m}{\ell}\binom{n}{m}$ for $m < s$ and 
$\binom{n-m}{\ell}\binom{n}{n-m} = \binom{n-m}{\ell}\binom{n}{m}$ 
for $m < t$, and the same for $n'$, so we
can apply Lemma \ref{binomial}.
\end{proof}

\subsubsection{Bases}
Let $A$ be one of the rings $\Z_p$, $R$, $\K$, $E$.
We define two different bases for the free $A$-module of rank $r+1$
$M_r(A) = A[X,Y]_r$. We omit the mention of $A$ if the statement does not
depend on it.
Let $\theta = X^pY-XY^p \in A[X,Y]$. Let $m(r) = \lfloor r/(p+1) \rfloor$, and 
$r = (p+1)m(r) + t(r)$, so $0 \leq t(r) \leq p$.

We introduce first a family of elements of $M_r(A)$.
We set $b_{n,i}(r) = \theta^nX^{r-n(p+1)-i}Y^i$ if $n \leq m(r)$ and $r
\geq n(p+1)+i$. We also set $b_{n,\infty}(r) = \theta^nY^{r-n(p+1)}$ if
$n \leq m(r)$.

Our first basis is the set $\B_0(r) = \{b_{0,i}(r), 0 \leq i \leq r\}$.
Our second basis is the set $\B_{\theta}(r) = \{b_{n,i}(r), n < m(r), 0
\leq i \leq p-1\text{ or }i = \infty\} \cup \{b_{m(r),i}(r), 0 \leq i < 
t(r)\text{ or }i = \infty \}$. 
The fact that $\B_{\theta}(r)$ is indeed a basis in a consequence
of an analogue of the computations of Section \ref{explicitbasis}.

Let us explain a little why we introduce this basis:
in order to compute the
reduction for $r$, we find a finite number of elements $f$ such that $(T-a_p)f$
is integral, and its reduction modulo $p$ gives us some information.
We now want to
find, for $r'$, a function $f'$ such that $(T-a_p)f'$ gives us analogous
information about the reduction. So we need some way to transform an
element of $\Sym^r \K^2$ into an element of $\Sym^{r'}\K^2$. This is why
we introduce $\B_{\theta}$: we transform the element $b_{n,i}(r)$ into
the element $b_{n,i}(r')$. However, we can not work only with the basis
$\B_{\theta}$, as the action of the operator $T$ is more naturally
expressed in the basis $\B_0$. So we need to understand how to go from
$\B_0$ to $\B_{\theta}$, and how this differs for $r$ and $r'$, which is
what we do in Lemmas \ref{bthetasmall} and \ref{bthetaend}.

For an element $f\in M_r$ and $n \leq m(r)$ and $j \in
\{0,\dots,p-1,\infty\}$, we set $\lambda_r(f ; n,j)$ for the coordinate
in $b_{n,j}(r)$ of $f$ written in the basis $\B_{\theta}(r)$.
When $f = b_{t,i}(r)$, we also write 
$\lambda_r(t,i ; n,j)$ for $\lambda_r(b_{t,i}(r) ; n,j)$.

For any integer $c \leq m(r)$, we denote by $M^{c}_r$ the submodule of $M_r$
of elements divisible by $\theta^c$, it is also the submodule of $M_r$
generated by the elements $b_{t,i}(r) \in \B_{\theta}(r)$ with $t \geq c$. 
We denote by $M_{r,c}$ be the submodule of $M_r$ generated by the
elements of the basis $\B_{\theta}(r)$ of the form $b_{n,i}(r)$ for $n <
c$. We have $M_r = M_r^c \oplus M_{r,c}$, and we denote by $\pi_{r,c} :
M_r \to M_{r,c}$ be the projection attached to this decomposition.
Note that the decomposition is compatible with change of rings (such as
the reduction $R \to E$ or the inclusion $R \to \K$).
If $r' > r$, we denote by $\psi_{r,r',c}$ the map $M_{r,c} \to M_{r',c}$
that sends $b_{n,i}(r)$ to $b_{n,i}(r')$ for all $n < c$.

\begin{lemm}
\label{computebtheta}
The following algorithm gives the coordinates of an element $z \in M_r$
in the basis $\B_{\theta}(r)$: 
start with the element $z$ written as any
linear combination of elements of the form $b_{n,j}(r)$ with
$n(p+1)+j \leq r$ or $j = \infty$. Then apply the following
set of transformations on each $b_{n,j}(r)$ appearing with a non-zero
coefficient:
\begin{enumerate}
\item 
each $b_{n,j}(r)$ with $j \leq p-1$ or $j = \infty$ is unchanged

\item 
each $b_{n,j}(r)$ with $r = n(p+1)+j$ is replaced by
$b_{n,\infty}(r)$

\item
each $b_{n,j}(r)$ with $p\leq j < r-n(p+1)$ is replaced by
$b_{n,j+1-p}(r) - b_{n+1,j-p}(r)$.
\end{enumerate}
After the set of transformations has been applied a finite number of
times, all the elements that appear are left unchanged ; 
at this point we have written
$z$ as a linear combination of elements of $\B_{\theta}(r)$.
\end{lemm}

\begin{proof}
It suffices to check that the transformations do not change the element
$z$ being represented by the linear combination, and that the elements
left unchanged are exactly the elements of $\B_{\theta}(r)$.
\end{proof}

\begin{lemm}
\label{bthetasmall}
Suppose $r'>r$. Let $t \leq  r-c$. Then for all $i \leq t$, and all $n < 
c$, and all $j \in \{0,\dots,p-1,\infty\}$, 
we have $\lambda_r(0,i;n,j) = \lambda_{r'}(0,i;n,j)$.
\end{lemm}

\begin{proof}
Start with $z = b_{0,i}(r)$, and apply the algorithm of Lemma
\ref{computebtheta}. Then: we will apply step (2) of the set of
transformations only to elements $b_{n,j}(r)$ with $n \geq c$. 
Indeed, consider the quantity $r-np-j$ for each $b_{n,j}(r)$ that appears
at some point of the algorithm. Then this quantity does never go down
during the application of the algorithm unless we apply step (2), 
and it is equal to $r-i \geq r-t
\geq c$ for $b_{0,i}(r)$. We apply step (2) only to a $b_{n,j}(r)$ for which
the quantity $r-np-j$ is equal to $n$, so only to some $b_{n,j}(r)$ with
$n \geq c$.

This property also holds for $z' = b_{0,i}(r')$ as $r' > r$. 

Now we start with $z = b_{0,i}(r)$ and $z' = b_{0,i}(r')$, and we write 
$z = \sum\mu_{n,j}^m b_{n,j}(r)$ and $z' = \sum{\mu_{n,j}^m}' b_{n,j}(r')$
the linear combination obtained after the $m$-th time we have applied the
set of transformations. We need only show that for all $m$, and all
$n\leq c$, we have $\mu_{n,j}^m={\mu_{n,j}^m}'$.

A difference can only appear when we apply step (2) of the set of
transformations, if we have $r = n(p+1)+j$ but $r' > n(p+1)+j$. 
But as we saw, it can happen only for $n\geq c$ so it does not play a role in
the $\lambda_r(0,i;n,j)$ and $\lambda_{r'}(0,i;n,j)$ for $n < c$.
\end{proof}

We have the following result:

\begin{lemm}
\label{crittheta}
Let $P = \sum_{i=0}^r\alpha_iX^{r-i}Y^i \in M_r$. Then $P \in M^{c}_r$ if
and only if the following conditions are satisfied:
\begin{enumerate}
\item
$\alpha_i = 0$ for all $i < c$

\item
$\alpha_i = 0$ for all $i > r-c$

\item
for all $a\in \Z/(p-1)\Z$ and all $\ell < c$ we have 
$\sum_{i \equiv a \mod p-1}\binom{i}{\ell}\alpha_i = 0$.
\end{enumerate}
\end{lemm}

\begin{proof}
We do the proof only in characteristic zero. The result still holds in
characteristic $p$, with the same proof except that the derivative is
replaced by the Hasse derivative.

Let $P = \sum_{i=0}^r\alpha_iX^{r-i}Y^i$, and set $f_P(t) = \sum_i\alpha_it^i$.
Fix $a$ and $\ell$, and let $g_{a,\ell,P}(t) =
t^{\ell-a}f^{(\ell)}_P(t)$. Then $g_{a,\ell,P}(t) =
\ell!\sum_i\binom{i}{\ell}\alpha_it^{i-a}$, so 
$\ell!(p-1)\sum_{i \equiv a\mod p-1}\binom{i}{\ell}\alpha_i 
= \sum_{\xi\in\mu_{p-1}}g_{a,\ell,P}(\xi)$.

Suppose that $\theta^c$ divides $P$. Conditions (1) and (2)
are clear as $X^cY^c$ divides $\theta^c$.  Moreover
$(t-t^p)^c$ divides $f_P$, so for all $\xi\in\mu_{p-1}$ we have that
$(t-\xi)^c$ divides $f_P$. So $(t-\xi)$ divides $f^{(\ell)}_P$, hence
$g_{a,\ell,P}$, for all $\ell < c$. So $g_{a,\ell,P}(\xi) = 0$ and hence
$\sum_{i \equiv a\mod p-1}\binom{i}{\ell}\alpha_i = 0$. So $P$ satisfies
also condition (3).

Suppose now that $P$ satisfies the conditions. Fix $\ell < c$. Then 
$\sum_{i \equiv a\mod p-1}\binom{i}{\ell}\alpha_i = 0$ for all $a$ implies
that $\sum_{\xi}\xi^{\ell-a}f^{(\ell)}_P(\xi) = 0$ for all $a$, so 
$f^{(\ell)}_P(\xi) = 0$ for all $\xi\in\mu_{p-1}$. As this is true for
all $\ell<c$, it means that $(1-t^{p-1})^c$ divides $f_P$. As $\alpha_i = 0$
for all $i<c$ we see that $(t-t^p)^c$ divides $f_P$. Finally $\deg f_P
\leq r-c$ from condition (2), so $\theta^c$
divides $P(X,Y) = X^rf_P(Y/X)$.
\end{proof}

\begin{lemm}
\label{modthetac}
There exists an invertible $A$-linear map $\psi: A^{c(p-1)} \to
A^{c(p-1)}$,
depending only on $c$,
with coordinates $\psi_0,\dots,\psi_{c(p-1)-1}$
such that for all 
$f = \sum_{i=c}^{r-c}f_ib_{0,i}(r) \in M_r(A)$, we have
$f-\tilde{f}\in M_r^c$ where $\tilde{f} =
\sum_{i=0}^{c(p-1)-1}\psi_i(v(f))b_{0,i+c}(r)$
and $v(f) = (\sum_{i\equiv a \mod p-1}\binom{i}{\ell}f_i)_{a \in
\Z/(p-1)\Z, 0 \leq \ell < c}$.
\end{lemm}

\begin{proof}
For all $a \in \Z/(p-1)\Z$, denote by $a_0$ the smallest integer
$\geq c$ which is congruent to $a$ modulo $p-1$. Then for all
$a \in \Z/(p-1)\Z$ 
consider the $\Z_p$-linear map $\phi_a : \Z_p^c \to \Z_p^c$,
$\phi_a(x_0,\dots,x_{c-1}) = 
(\sum_{j=0}^{p-1}\binom{a_0+j(p-1)}{\ell}x_j)_{0\leq \ell <c}$.
Then the map $\phi_a$ is invertible by Lemma \ref{determinant}.
We put all the maps $\phi_a$ together for $a \in \Z/(p-1)\Z$, we get
a map $\phi : \Z_p^{c(p-1)} \to \Z_p^{c(p-1)}$ such that
$\phi(x_0,\dots,x_{c(p-1)-1})$ is equal to 
$(\sum_{j=0}^{p-1}\binom{a_0+j(p-1)}{\ell}x_j)_{0\leq \ell <c,a\in \{0,\dots,p-2\}}$.
Denote by $\psi$ the inverse of $\phi$.

Let $\tilde{f}$ be as in the statement. Then by
construction, $f-\tilde{f}$ satisfies the conditions of Lemma \ref{crittheta}
so $f-\tilde{f} \in M_r^c$.
\end{proof}

\begin{lemm}
\label{bthetasum}
Let $d > 0$. Suppose that $r' > r$.
Let $f \in M_r(R)$ and $f' \in M_{r'}(R)$ with 
$f = \sum_{i = c}^{r-c}f_ib_{0,i}(r)$ and
$f = \sum_{i = c}^{r'-c}f'_ib_{0,i}(r')$.
Suppose that for all $a\in \Z/(p-1)\Z$ and for all $\ell < c$ we
have 
$\sum_{i\equiv a \mod  p-1}\binom{i}{\ell}f_i 
\equiv
\sum_{i\equiv a \mod  p-1}\binom{i}{\ell}f'_i 
\mod p^d$.
Then $\lambda_r(f;n,j) \equiv \lambda_{r'}(f';n,j) \mod p^d$ for all $n < c$ and $j \in
\{0,\dots,p-1,\infty\}$.
\end{lemm}

\begin{proof}
Let $\tilde{f}$ and $\tilde{f'}$ be attached to $f$ and $f'$ as in Lemma
\ref{modthetac}. 
Then $\lambda_r(f ; n,j) = \lambda_r(\tilde{f} ; n,j)$ and 
$\lambda_{r'}(f' ; n,j) = \lambda_r(\tilde{f'} ; n,j)$ for all $n<c$.
Moreover, let $\tilde{f} = \sum\alpha_ib_{0,i}(r)$ and
$\tilde{f} = \sum\alpha'_ib_{0,i}(r')$.
Then $\alpha_i = 0$ and $\alpha_i' = 0$ for all $i \geq c(p-1)$, 
and $\alpha_i \equiv \alpha_i' \mod p^d$ for all $i < c(p-1)$ by
construction. So we can conclude by Lemma \ref{bthetasmall}.
\end{proof}

\begin{lemm}
\label{bthetamiddle}
Let $d > 0$. Suppose that $r' > r$ and
$r' \equiv r \mod (p-1)p^{d+\lfloor\log_p(c)\rfloor}$.
Let $c \leq i \leq pc$. Then
coefficients $\lambda_r(0,r-i ; n,j)$ and $\lambda_{r'}(0,r'-i ;
n,j)$ are congruent modulo $p^d$ for all $n < c$ and $j \in
\{0,\dots,p-1,\infty\}$.
\end{lemm}

\begin{proof}
Let $f = b_{0,r-i}(r)$ and $f' = b_{0,r'-i}(r')$.
We want to apply Lemma \ref{bthetasum} to $f$ and $f'$.
The hypotheses are satisfied thanks to 
Lemma \ref{binomial}.
Indeed the sums considered 
either $0$ for both $f$ and $f'$, if we do not sum over the correct
congruence class
or $\binom{r-i}{\ell}$ for some $\ell < c$ for $f$
and  $\binom{r'-i}{\ell}$ for $f'$.
\end{proof}

\begin{lemm}
\label{bthetaend}
Let $d > 0$.
Suppose $i < c$ and $r' > r$ and $r' \equiv r \mod
(p-1)p^{d+\lfloor\log_p(c)\rfloor}$.
Then the coefficients $\lambda_r(0,r-i ; n,j)$ and $\lambda_{r'}(0,r'-i ;
n,j)$ are congruent modulo $p^d$ for all $n < c$ and $j \in
\{0,\dots,p-1,\infty\}$.
\end{lemm}

\begin{proof}
Observe that for all $t < c$ we have $b_{0,r-t}(r) =
(-1)^t(b_{t,\infty}(r) -
\sum_{j=0}^{t-1}\binom{t}{j}(-1)^jb_{0,r-tp+j(p-1)})$.
We rewrite $b_{0,r-i}(r)$ as follows: start with 
$b_{0,r-i}(r)$. At each step, replace each term of the form
$b_{0,r-t}(r)$ with $t < c$ by the formula above. Stop when we have
written $b_{0,r-i}(r)$ as $\sum_{j=0}^{c-1}x_jb_{j,\infty}(r) + 
\sum_{t=c}^{p(c-1)}y_tb_{0,r-t}(r)$. We can do the same thing for
$b_{0,r'-i}(r')$, and it is easily checked that we get the same formula
as for $b_{0,r-i}(r)$, that is: $b_{0,r'-i}(r') =
\sum_{j=0}^{c-1}x_jb_{j,\infty}(r') + 
\sum_{t=c}^{p(c-1)}y_tb_{0,r'-t}(r')$ with the same coefficients $x_j$ and
$y_t$. 
Now it suffices to use Lemma \ref{bthetamiddle} to conclude.
\end{proof}

\subsubsection{Action of $T$}
\label{Tlocconst}

Let $c$ be an integer with $c \leq r/(p+1)$ and
let $d$ be an integer with $c \geq d$.

\begin{prop}
\label{transfoT}
Let $f \in I(M_{r}(R))$. Suppose that $r' \geq r$ and $r' \equiv r \mod
(p-1)p^{d+\lfloor\log_p(c)\rfloor}$. Then
$\pi_{r',c}(T-a_p)\psi_{r,r',c}\pi_{r,c}(f)$ is equal to 
$\psi_{r,r',c}\pi_{r,c}(T-a_p)f$ modulo $p^d$.
\end{prop}

\begin{proof}
It is enough to prove this for $a_p = 0$, as the equality for the part
with $-a_pf$ is clear. By linearity it is enough to prove it for $f$ of
the form $[g^\eps_{n,\mu},b_{t,i}]$ for some $t \leq c$, and $0 \leq i \leq
p-1$ or $i = \infty$.

We do the proof for the $g^0_{n,\mu}$, the case of $g^1_{n,\mu}$ is
similar. Recall that explicit formulas for $T^+$ and $T^-$ are given in Paragraph 
\ref{descrT}.
We fix $g^0_{n,\mu}$. Recall from the notations of Paragraph
\ref{descrT} that we write $\mu = \underline{\mu} + p^{n-1} \mu'$ with
$\underline{\mu}\in I_{n-1}$.
For each $f \in M_r$ and $u \in I_1$, let $P_{r,u}(f)$ be the
polynomials such that $T^+[g^0_{n,\mu},f] = \sum_{u\in
I_1}[g^0_{n,\mu}g^0_{1,u},P_{r,u}(f)]$, and $P_{r,-}(f)$ be the
polynomial such that 
$T^-[g^0_{n,\mu},f] = [g^0_{n-1,\underline\mu},P_{r,-}(f)]$ (or 
$T^-[g^0_{0,0},f] = [g^1_{0,0},P_{r,-}(f)]$). 
Note that $P_{r,-}$ depends on $g^0_{n,\mu}$, but $P_{r,u}$ does not.
We also define analogous polynomials for $f\in M_{r'}$.

So what we have to do is: 
check that for all $t,n < c$ and $i,j \in \{0,\dots,p-1,\infty\}$, for
all $u \in I_1 \cup \{-\}$, we have
$\lambda_r(P_{r,u}(b_{t,i}(r)) ; n,j) \equiv
\lambda_{r'}(P_{r',u}(b_{t,i}(r'));n,j) \mod p^d$.

\bigskip

Start with for $b_{t,i}(r)$ for $0 \leq i \leq p-1$. 
We have
$b_{t,i} = \sum_{\ell=0}^t\binom{t}{\ell}(-1)^\ell
b_{0,t+i+\ell(p-1)}$ 

Let us treat the case of $u\in I_1$.
We notice that $P_{r,u}(b_{t,i}(r))$ and $P_{r',u}(b_{t,i}(r'))$ can both be written as
$\sum_{m\geq 0}p^ma_mb_{0,m}$ (with the same $a_m\in\Z_p$ for $r$ and $r'$).
As we are computing modulo $p^d$, with $d < c$, we 
get that $\lambda_r(P_{r,u}(b_{t,i}(r)) ; n,j) \equiv 
\sum_{m=0}^{d-1} p^ma_m\lambda_r(0,m;n,j) \mod p^d$
and the same for $r'$. Then we can make use of Lemma \ref{bthetasmall} to
conclude.

Consider now what happens with $P_{r,-}$ and $P_{r',-}$. We note that
$P_{r,-}(b_{0,n}(r))$ is a multiple of $p^{r-n}$, so if $n\leq r-d$ then 
$P_{r,-}(b_{0,n}(r))$ disappears modulo $p^d$. But in $b_{t,i}(r)$, the
largest $n$ such that $b_{0,n}(r)$ appears is $n = i+pt < cp \leq r-d$, so
$P_{r,-}(b_{t,i}(r)) \equiv 0 \mod p^d$. By the same reasoning, we
have
$P_{r',-}(b_{t,i}(r')) \equiv 0 \mod p^d$.

\bigskip

Consider now 
$b_{t,\infty}(r) = \sum_{\ell=0}^t\binom{t}{\ell}(-1)^\ell b_{0,r-tp+\ell(p-1)}$
and
$b_{t,\infty}(r') = \sum_{\ell=0}^t\binom{t}{\ell}(-1)^\ell
b_{0,r'-tp+\ell(p-1)}$.
We have that $P_{r,u}(b_{t,\infty}(r)) =
\sum_{\ell=0}^t\binom{t}{\ell}(-1)^\ell P_{r,u}(b_{0,r-tp+\ell(p-1)})$.

Let us treat the case of $u\in I_1$.
The coefficient of $b_{0,j}(r)$
in $P_{r,u}(b_{t,\infty}(r))$ is
$$
p^j\sum_{\ell=0}^t\binom{t}{\ell}(-1)^\ell
\binom{r-tp+\ell(p-1)}{j}(-u)^{r-tp+\ell(p-1)-j}
$$
and the formula for the coefficient of $b_{0,j}(r')$ in 
$P_{r',u}(b_{t,\infty}(r'))$ is the same with $r$ replaced by $r'$.
Let us show that they are congruent modulo $p^d$
(we are only interested in the case where $j < d$): in both cases, the
$-u$ part is elevated to powers that are congruent modulo $p-1$ (and always non-zero), 
as $r' \equiv r \mod p-1$, so
$(-u)^{r-tp+\ell(p-1)-j} = (-u)^{r'-tp+\ell(p-1)-j}$. We also have that
$\binom{r-tp+\ell(p-1)}{j} \equiv \binom{r'-tp+\ell(p-1)}{j} \mod p^d$
by Lemma \ref{binomial}.

Consider now what happens when $u = -$. 
We see that
$P_{r,-}(b_{0,r-tp+\ell(p-1)})$ is divisible by $p^{tp-\ell(p-1)}$, and
the same for $r'$, so we are only interested in the terms for which
$tp-\ell(p-1) < d$. 

Suppose first that either $(n,\mu) = (0,0)$ (so $g^0_{n,\mu} = 1$) or
$\mu' = 0$,
then $P_{r,-}(b_{0,r-i}(r)) = p^ib_{0,r-i}(r)$. So we can apply
Lemma \ref{bthetaend}, as we consider only $i$ that are $< d \leq c$.

Suppose now that $\mu' \neq 0$, so 
$P_{r,-}(b_{0,i}(r)) =
p^{r-i}\sum_{j=0}^i\binom{i}{j}{\mu'}^{i-j}b_{0,j}(r)$.
In particular
$P_{r,-}(b_{0,r-s}(r)) =
p^s\sum_{j=0}^{r-s}\binom{r-s}{j}{\mu'}^{j}b_{0,r-s-j}(r)$
and we will apply this to $s = tp-\ell(p-1)<d$.

We split 
$P_{r,-}(b_{0,r-s}(r))$ in two parts: one with the $b_{0,r-s-j}(r)$ with
$s+j < c$, and one with the $b_{0,r-s-j}(r)$ with
$s+j \geq c$. We can apply Lemma \ref{bthetaend} to the first part, using
the fact that $\binom{r-s}{j} \equiv \binom{r'-s}{j} \mod p^d$ by
Lemma \ref{binomial}.
Consider now the second part.
We want to apply Lemma \ref{bthetasum} to it. For this, we have to
check that the sums
$p^s\sum_{r-s-j \equiv a \mod p-1, c \leq r-s-j \leq r-c}
\binom{r-s-j}{\ell}\binom{r-s}{j}{\mu'}^{j}$
and
$p^s\sum_{r'-s-j \equiv a \mod p-1, c \leq r'-s-j \leq r-c}
\binom{r'-s-j}{\ell}\binom{r'-s}{j}{\mu'}^{j}$
are congruent modulo $p^d$
for all $a \in \{0,\dots,p-2\}$ and $0 \leq \ell < c$.
Denote these sums by 
$\Sigma_{r,s,a,\ell}$ and
$\Sigma_{r',s,a,\ell}$ respectively.
Then we can write
$\Sigma_{r,s,a,\ell} = p^s{\mu'}^a\sum_{c \leq i \leq r-c, i \equiv r-s-a
\mod p-1}\binom{i}{\ell}\binom{r-s}{i}$, and similarly for
$\Sigma_{r',s,a,\ell}$. So we can conclude by Lemma \ref{sumbinom}.
\end{proof}

\subsubsection{Proof of Theorem \ref{locconstr}}
\label{prooflocconstr}

Let $k$, $a_p$, $\delta$, and $c$ be as in the statement of
Theorem \ref{locconstr}. Let $d$ be the smallest integer $> \delta$, that
is,
$d = 1 + \lfloor \delta \rfloor$.  Let $r = k-2$, $k' > k$ with $k' \equiv k
\mod (p-1)p^{d+\lfloor\log_p(c)\rfloor}$ and $r' = k'-2$.
By hypothesis we have $d \leq c$ and $c \leq r/(p+1)$ and so the results of
Paragraph \ref{Tlocconst} apply. 

Let $V$ be a subspace of $M_r$. We denote by $I(V)$ the subspace of
$I(M_r)$ generated by all elements of the form $[g^\eps_{n,\mu},v]$ for
$v \in V$. In general it is not a subrepresentation of $I(M_r)$.

\smallskip

Recall that in order to compute $\bar\Theta_{k,a_p}$, we have introduced
the standard filtration $(V_i(r))_i$ of $M_r(E)$ defined in Section \ref{filtration}. 
The computations give us
information about the image of $I(J_i(r))$ inside
$\bar\Theta_{k,a_p}$, where $J_i(r) = V_i(r)/V_{i+1}(r)$. 
Using Corollary
\ref{filtrzero}, we see that we need only to study $I(J_i(r))$
for $0 \leq i \leq 2\lfloor v_p(a_p) \rfloor + 1$. 
Note also that $M_r^c = V_{2c}(r)$, so for all $i$ in this range, we have $M_r^c \subset
V_{i+1}(r)$, and $\pi_{r,c}(V_i(r)) \subset V_i(r)$. 
In particular, the image of an element of $I(V_i(r))$ in
$I(J_i(r))$ depends only on is projection to $I(M_{r,c})$.

All the similar statements hold for $r'$. Note also that for all $i$ in
the interesting range,
$J_i(r)$ and $J_i(r')$ are isomorphic as $r \equiv r' \mod p-1$. We
fix the isomorphism given by the bases described in
Paragraph \ref{explicitbasis}. Moreover, for all $x \in M_{r,c} \cap
V_i(r)$, the image of $x\in J_i(r)$ and the image of $\psi_{r,r',c}(c)$
in $J_i(r')$ correspond by this isomorphism.

Let $i\leq 2\lfloor v_p(a_p) \rfloor + 1$. 
The information we have found about $I(J_i(r))$
has the following form: we have an element $\phi_i \in I(M_r(\K))$, such
that $(T-a_p)\phi_i$ is integral. Let $u_i$ be the image of
$(T-a_p)\phi_i$ in $I(M_r(E))$. Then $u_i \in I(V_i(r))$, and its image
in $I(J_i(r))$  is equal to
$v_i$ or $w_{i,1}-\lambda_iv_i$ or $w_{i,2}-\mu_iw_{i,1}+v_i$ with the
notation of Paragraph \ref{explicitF}.

The hypothesis about the denominators means that we can take $\phi_i$ to
be of the form $p^{-\delta}\psi_i$ with $\psi_i$ integral.

Our goal is now to find elements $\phi_i' \in
I(M_{r'}(\K))$ such that 
$(T-a_p)\phi'_i$ is integral and has its reduction $u_i'$ modulo $p$
in $I(V_i(r'))$, with the same image as $u_i$ in $I(J_i(r'))$ (via the
isomorphism between $J_i(r)$ and $J_i(r')$).
Indeed, this will mean that for any information
that we have about the image of $I(J_i(r))$, we have the same information about
$I(J_i(r'))$, so we can compute 
$\bar\Theta_{k',a_p}$ from this information if we can compute
$\bar\Theta_{k,a_p}$, and they are the same.

We claim that we can take $\phi_i'$ of the form
$\psi_{r,r',c}(\pi_{r,c}(\phi_i)) + y_i$ with $y_i \in
I(M^{c}_{r'}(\K))$. Denote $\psi_{r,r',c}(\pi_{r,c}(\phi_i))$ by
$\tilde{\phi}_i$.

Let $z_i =  \pi_{r,c}(\psi_i)-\psi_i$, so that
$(T-a_p)\pi_{r,c}(\psi_i) = (T-a_p)\psi_i + (T-a_p)z_i$.
Note that $z_i \in I(M_r^{c})$. By Lemma \ref{Ttheta},
using the fact that $c > \delta$, we get that $Tp^{-\delta}z_i$ is integral
and reduces to zero modulo $p$. As $\pi_{r,c}(-a_pz_i) = 0$, we get
that $\pi_{r,c}(T-a_p)\pi_{r,c}(\phi_i)$ is integral, and
has the same image as $u_i$ in $I(J_i(r))$. We write $\tilde{u}_i$ for
the reduction modulo $p$ of $\pi_{r,c}(T-a_p)\pi_{r,c}(\phi_i)$.

By the results of Paragraph \ref{Tlocconst}, we get that
$\pi_{r',c}((T-a_p)\tilde\phi_i)$ is integral and reduces to
$\psi_{r,r',c}(\tilde{u}_i)$
modulo $p$. Let $z'_i = (T-a_p)\tilde\phi_i-\pi_{r',c}(T-a_p)\tilde\phi_i$.
Then $z_i'$ is in $p^{-\delta}I(M^{c}_{r'})$, as $\tilde\phi_i$ is in
$p^{-\delta}I(M_{r'}(R))$. So we can find $y_i$ such that
$(T-a_p)y_i-z_i'$ in integral and reduces to zero modulo $p$. Indeed, it
is enough to show that we can do this for $z_i'$ for the form
$p^{-\delta}[g,\theta^c f]$ for some integral $f$. But then $y_i =
-(p^{-\delta}/a_p)[g,\theta^c f]$ works, as $v_p(p^{\delta}a_p) < c$.  
Consider now $\phi_i'=\tilde\phi_i+y_i$, then $(T-a_p)\phi_i'$ is integral and 
has the same reduction modulo $p$ as $\pi_{r',c}(T-a_p)\tilde\phi_i$, that is
it reduces to $\psi_{r,r',c}(\tilde{u}_i)$ modulo $p$. So the image
of the reduction modulo $p$ of $(T-a_p)\phi_i'$ in $I(J_i(r'))$ is
the same as the image of $(T-a_p)\phi_i$ in $I(J_i(r))$, which is what we
wanted.

\section{Examples}
\label{examples}

We give here some examples of reductions computed thanks to our
implementation in SAGE (\cite{sage}) of the algorithm, which can be found at \cite{algo}. 
We also give some guess as to the value of some reductions: based on the
computations, we try to give the simplest possible description of the
locus of $a_p$'s with a given reduction that is compatible with the
computed data. 
A technique to prove whether the description of the locus is correct (for a
fixed $p$, $k$, and residual representation) is explained in
\cite{rig}.


The tables give the value of $p$; $k$; $a_p$; the steps $F_i$ of
the filtration as described in Section \ref{computefromrel} that can be
non-zero, and their value; the semi-simplification of the reduction; the
largest elementary divisor $\delta$ that appeared in the computation that led to
the result (so we get a radius for the local constancy result of
Corollary \ref{locconstapbis}, and, when applicable, a radius for the
local constancy result of Corollary \ref{locconstrbis}); 
and finally the value of $n$ that allowed to determine the reduction.

\subsection{Examples in slope $1$}

The case where $v(a_p) = 1$ has been studied in \cite{BGR} for $p > 3$. In that
article, a complete description of the reduction is given for $p \geq 5$, 
except in the case where $k$ is equal to $4$ modulo $p(p-1)$ and
$v(a_p^2-\binom{k-2}{2}p^2) > 2$. We give some
results for this case, for $a_p^2$ close to
$p\binom{k-2}{2}$. We recall that in the case where
$k$ is equal to $4$ modulo $p(p-1)$ and 
$v(a_p^2-\binom{k-2}{2}p^2) = 2$ then $\bar{V}_{k,a_p}$ is irreducible,
isomorphic to $\ind\omega_2^3$, and more precisely $\bar\Theta_{k,a_p}$
is isomorphic to $\pi(p-3,0,\omega^2)$, which is the $F_0$-part of the
filtration as described in Sections \ref{computefromrel} and
\ref{filtration} (beware that the numbering of the filtration is not the
same in this article and in \cite{BGR}).

Our guess as to the general form of the reduction in this case is as
follows:
Suppose that $p > 3$, $k$ is equal to $4$ modulo $p(p-1)$, and $k > 2$. There exists an
integer $i(k) > 1$ such that:
\begin{itemize}
\item
if $v(a_p) = 1$ and $v(a_p^2-\binom{k-2}{2}p^2) < i(k)$ then
$\bar{V}_{k,a_p}$ is isomorphic to $\ind\omega_2^3$.
\item
if $v(a_p) = 1$ and $v(a_p^2-\binom{k-2}{2}p^2) = i(k)$ then
$\bar{V}_{k,a_p}^{ss}$ is isomorphic to
$\omega^2\unr(\lambda)\oplus\omega\unr(\lambda^{-1})$ for some
$\lambda\in E^\times$.
\item
if $v(a_p) = 1$ and $v(a_p^2-\binom{k-2}{2}p^2) > i(k)$ then 
$\bar{V}_{k,a_p}$ is isomorphic to $\ind\omega_2^{p+2}$.
\end{itemize}
Note that the reduction for $v(a_p) = 1$ and $k$ is equal to $4$ modulo
$(p-1)$ is necessarily one of these three possibilities, as follows from
the results of \cite{BGR}, so it remains only
to understand what is the locus corresponding to each possibility.

\noindent
\begin{tabular}{|c|c|c|c|c|c|c|}
\hline
$p$ & $k$ & $a_p$ & non-zero $F_i$s &  $\bar{V}_{k,a_p}^{ss}$ & $\delta$ & $n$ \\
\hline
$5$ & $24$ & $2\cdot 5$ & $F_0$: quotient of $I(\sigma_2(2))$ & $\ind\omega_2^3$ & 2 & 2 \\
\hline
$5$ & $24$ & $5\sqrt{11\cdot 21}$ & $F_3$: $\pi(4,0,\omega)$ & $\ind\omega_2^7$ & 3 & 3 \\
\hline
$5$ & $24$ & $5\sqrt{11\cdot 21}+5^2$ & $F_3$: $\pi(4,4,\omega)$ and $F_0$: $\pi(2,4,\omega^2)$ & $\omega\unr(-1)\oplus\omega^2\unr(-1)$ & 3 & 3 \\
\hline
$5$ & $24$ & $5\sqrt{11\cdot 21}+2\cdot 5^2$ & $F_3$: $\pi(4,3,\omega)$ and $F_0$: $\pi(2,2,\omega^2)$ & $\omega\unr(2)\oplus\omega^2\unr(3)$ & 3 & 3 \\
\hline
$5$ & $24$ & $5\sqrt{11\cdot 21}+3\cdot 5^2$ & $F_3$: $\pi(4,2,\omega)$ and $F_0$: $\pi(2,3,\omega^2)$ & $\omega\unr(3)\oplus\omega^2\unr(2)$ & 3 & 3 \\
\hline
$5$ & $24$ & $5\sqrt{11\cdot 21}+4\cdot 5^2$ & $F_3$: $\pi(4,1,\omega)$ and $F_0$: $\pi(2,1,\omega^2)$ & $\omega\oplus\omega^2$ & 3 & 3 \\
\hline
\hline
$5$ & $44$ & $5\sqrt{21\cdot 41}$ & $F_3$: $\pi(4,0,\omega)$ & $\ind\omega_2^7$ & 3 & 3 \\
\hline
$5$ & $44$ & $5\sqrt{21\cdot 41}+5^2$ & $F_3$: $\pi(4,2,\omega)$ and $F_0$: $\pi(2,3,\omega^2)$ & $\omega\unr(3)\oplus\omega^2\unr(2)$ & 3 & 3 \\
\hline
$5$ & $44$ & $5\sqrt{21\cdot 41}+2\cdot 5^2$ & $F_3$: $\pi(4,4,\omega)$ and $F_0$: $\pi(2,4,\omega^2)$ & $\omega\unr(4)\oplus\omega^2\unr(4)$ & 3 & 3 \\
\hline
$5$ & $44$ & $5\sqrt{21\cdot 41}+3\cdot 5^2$ & $F_3$: $\pi(4,1,\omega)$ and $F_2$: $\pi(0,1,\omega)$ and$F_0$: $\pi(2,1,\omega^2)$ & $\omega\oplus\omega^2$ & 3 & 3 \\
\hline
$5$ & $44$ & $5\sqrt{21\cdot 41}+4\cdot 5^2$ & $F_3$: $\pi(4,3,\omega)$ and $F_0$: $\pi(2,2,\omega^2)$ & $\omega\unr(2)\oplus\omega^2\unr(3)$ & 3 & 3 \\
\hline
\hline
$5$ & $104$ & $5\sqrt{51\cdot 101}$ & $F_3$: quotient of $I(\sigma_4(1))$ and $F_2$: quotient of $\pi(0,0,\omega)$ & $\ind\omega_2^7$ & 3 & 4 \\
\hline
$5$ & $104$ & $5\sqrt{51\cdot 101}+5^3$ & $F_3$: $\pi(4,4,\omega)$ and $F_0$: $\pi(2,4,\omega^2)$ & $\omega\unr(4)\oplus\omega^2\unr(4)$ & 4 & 4 \\
\hline
\end{tabular}

\subsection{Examples in slope $2$}

Based on the results in slope $1$, we can expect that when $v(a_p) = 2$
then in most cases the reduction is reducible, of the form
$\omega^{k-3}\unr(\lambda) \oplus \omega^{2}\unr(\lambda^{-1})$ for some
$\lambda \in E^\times$. Computations with small values of $p$ ($p \leq
7$) makes it difficult to see if this case is indeed the most frequent,
and in our examples we have found many different types of reduction.

We give one example where the form of the reduction is particularly
complicated and several different types of reduction appear for the same
$p$ and $k$. We take $p=7$ and $k=48$, then our computations are
compatible with the following description:

\begin{itemize}
\item
If $v(a_7-7^2-5\cdot 7^3)) = 2$ 
and 
$v(a_7+7^2+5\cdot 7^3) = 2$
then  
$\bar{V}_{48,a_7}^{ss}$ is isomorphic to
$\omega^4\unr(\lambda)\oplus\omega\unr(\lambda^{-1})$ for some
$\lambda\in E^\times$.

\item
If 
$2 < v(a_7-7^2-5\cdot 7^3) < 3$ 
or
$2 < v(a_7+7^2+5\cdot 7^3) < 3$ 
then  
$\bar{V}_{48,a_7}^{ss}$ is isomorphic to $\ind\omega_2^{11}$.

\item
If 
$v(a_7-7^2-5\cdot 7^3) = 3$ 
or 
$v(a_7+7^2+5\cdot 7^3) = 3$ 
then  
$\bar{V}_{48,a_7}^{ss}$ is isomorphic to
$\omega^3\unr(\lambda)\oplus\omega^2\unr(\lambda^{-1})$ for some
$\lambda\in E^\times$.

\item
If 
$v(a_7-7^2-5\cdot 7^3) > 3$ 
or
$v(a_7+7^2+5\cdot 7^3) > 3$ 
then  
$\bar{V}_{48,a_7}^{ss}$ is isomorphic to $\ind\omega_2^{17}$.
\end{itemize}

\subsection{Examples in slope $3/2$}

To describe the reduction modulo $p$ in this case, we introduce a
notation: let $I_{n,c}$ be the representation $\bar{\rho}$ of $G_{\Q_p}$ with
restriction to inertia isomorphic to $\omega^n\oplus\omega^n$, with
determinant equal to $\omega^{2n}$, and such
that the Frobenius acting on $\omega^{-n}\bar\rho$ has trace $c$.

The case where $1 < v(a_p) < 2$ was computed for $p > 2$ in \cite{BG15},
except for the special case where $k = 5 \mod{p-1}$ and 
$v(a_p^2 - (k-4)^2\frac{k-3}{2}p^3) > 3$. We give some examples of what happens in this
special situation.

\subsubsection{First case}
\label{32case1}

The first case is a situation that looks very similar to what happens in
slope $1/2$ for $k = 3 \mod{(p-1)}$, as described in \cite{BG13}. Our
guess is that when $p$ does not divise $k-5$, then the only step of the
filtration that is non-zero is $F_2$, and we have a rational number $i(k)$ such that:
\begin{itemize}
\item
if $1 < v(a_p) < 2$ and $v(a_p^2-(k-4)^2\frac{k-3}{2}p^3) < i(k)$ then
$\bar\Theta_{k,a_p}^{ss}$ is isomorphic to $\pi(p-2,0,\omega^2)$, and so
$\bar{V}_{k,a_p}^{ss}$ is isomorphic to $\ind\omega_2^{1+3p}$.
\item
if $1 < v(a_p) < 2$ and $v(a_p^2-(k-4)^2\frac{k-3}{2}p^3) = i(k)$ then
$\bar\Theta_{k,a_p}^{ss}$ is isomorphic to
$I(\sigma_{p-2}(2)/(T^2-cT+1)$ for some $c\in E^\times$,
that is, $\bar{V}_{k,a_p}^{ss}$ is isomorphic to $I_{2,c}$.
\item
if $1 < v(a_p) < 2$ and $v(a_p^2-(k-4)^2\frac{k-3}{2}p^3) > i(k)$ then 
$\bar\Theta_{k,a_p}^{ss}$ is isomorphic to $I(\sigma_{p-2}(2))/(T^2+1)$,
that is, $\bar{V}_{k,a_p}^{ss}$ is isomorphic to $I_{2,0}$.
\end{itemize}
Note that is was proved in \cite{BG15} that the reduction has one of
the listed forms, so it remains only to understand the locus where each
reduction occurs.

\smallskip

\noindent
\begin{tabular}{|c|c|c|c|c|c|c|}
\hline
$p$ & $k$ & $a_p$ & non-zero $F_i$s &  $\bar{V}_{k,a_p}^{ss}$ & $\delta$ & $n$ \\
\hline
$5$ & $17$ & $5\sqrt{5}\cdot 13\sqrt{7}$ & $F_2$: $I(\sigma_3(2))/(T^2+1)$ & $I_{2,0}$ & 5/2 & 3 \\
\hline
$5$ & $17$ & $5\sqrt{5}\cdot 13\sqrt{7}+5^2$ & $F_2$:
$I(\sigma_3(2))/(T^2-4T+1)$ & $I_{2,4}$ & 5/2 & 3 \\
\hline
$5$ & $17$ & $5\sqrt{5}\cdot 13\sqrt{7}+2\cdot 5^2$ & $F_2$: $I(\sigma_3(2))/(T^2-3T+1)$ & $I_{2,3}$ & 5/2 & 3 \\
\hline
$5$ & $17$ & $5\sqrt{5}\cdot 13\sqrt{7}+3\cdot 5^2$ & $F_2$: $I(\sigma_3(2))/(T^2-2T+1)$ & $I_{2,2}$ & 5/2 & 3 \\
\hline
$5$ & $17$ & $5\sqrt{5}\cdot 13\sqrt{7}+4\cdot 5^2$ & $F_2$: $I(\sigma_3(2))/(T^2-T+1)$ & $I_{2,1}$ & 5/2 & 3 \\
\hline
\end{tabular}

\subsubsection{Second case}
\label{32case2}

The previous situation can not happen when $p$ divides $k-5$, as in this case the
results of \cite{BG15} show that when  
$v(a_p^2 - (k-4)^2\frac{k-3}{2}p^3) = 3$ the reduction is equal to
$\ind\omega_2^4$, not $\ind\omega_2^{1+3p}$ .
So we can expect the situation to be a little different here. Our guess
is that that we have two rational numbers $i(k) < j(k)$ such that:
\begin{itemize}
\item
if $1 < v(a_p) < 2$ and $v(a_p^2-(k-4)^2\frac{k-3}{2}p^3) < i(k)$ then
$\bar{V}_{k,a_p}$ is isomorphic to $\ind\omega_2^4$.
\item
if $1 < v(a_p) < 2$ and $v(a_p^2-(k-4)^2\frac{k-3}{2}p^3) = i(k)$ then 
$\bar{V}_{k,a_p}$ is isomorphic to 
$\omega^3\unr(\lambda)\oplus\omega\unr(\lambda^{-1})$
for some $\lambda\in E^\times$.
\item
if $1 < v(a_p) < 2$ and $i(k) < v(a_p^2-(k-4)^2\frac{k-3}{2}p^3) < j(k)$ then
$\bar{V}_{k,a_p}^{ss}$ is isomorphic to
$\ind\omega_2^{1+3p} $
\item
if $1 < v(a_p) < 2$ and $v(a_p^2-(k-4)^2\frac{k-3}{2}p^3) = j(k)$ then 
$\bar{V}_{k,a_p}$ is isomorphic to $I_{2,c}$
for some $c\in E^\times$.
\item
if $1 < v(a_p) < 2$ and $v(a_p^2-(k-4)^2\frac{k-3}{2}p^3) > j(k)$ then
$\bar{V}_{k,a_p}^{ss}$ is isomorphic to
$I_{2,0}$.
\end{itemize}

\noindent
\begin{tabular}{|c|c|c|c|c|c|c|}
\hline
$p$ & $k$ & $a_p$ & non-zero $F_i$s &  $\bar{V}_{k,a_p}^{ss}$ & $\delta$ & $n$ \\
\hline
$5$ & $25$ & $5\sqrt{5}\cdot 21\sqrt{11}$ & $F_2$: $I(\sigma_3(2))/(T^2+1)$ & $I_{2,0}$ & 7/2 & 3 \\
\hline
$5$ & $25$ & $5\sqrt{5}\cdot 21\sqrt{11}+5^3$ & $F_2$: $I(\sigma_3(2))/(T^2-T+1)$ & $I_{2,1}$ & 7/2 & 3 \\
\hline
$5$ & $25$ & $5\sqrt{5}\cdot 21\sqrt{11}+5^2\sqrt{5}\sqrt{11}$ & $F_2$: $\pi(3,0,\omega^2)$ & $\ind\omega_2^{p+3}$ & 7/2 & 3 \\ \hline
$5$ & $25$ & $5\sqrt{5}\cdot 21\sqrt{11}+5^2$ & $F_3$: $\pi(1,1,\omega)$ and $F_0$: $\pi(1,1,\omega^3)$ & $\omega^3\oplus\omega$ & 7/2 & 3 \\
\hline
\hline
$7$ & $47$ & $7\sqrt{7}\cdot 43\sqrt{22}$ & $F_2$: $I(\sigma_5(2))/(T^2+1)$ & $I_{2,0}$ & 7/2 & 3 \\
\hline
$7$ & $47$ & $7\sqrt{7}\cdot 43\sqrt{22}+7^3$ & $F_2$: $I(\sigma_5(2))/(T^2-3T+1)$ & $I_{2,3}$ & 7/2 & 3 \\
\hline
$7$ & $47$ & $7\sqrt{7}\cdot 43\sqrt{22}+7^2\sqrt{7}\sqrt{22}$ & $F_2$: $\pi(5,0,\omega^2)$ & $\ind\omega_2^{22}$ & 7/2 & 3 \\
\hline
$7$ & $47$ & $7\sqrt{7}\cdot 43\sqrt{22}+7^2$ & $F_3$: $\pi(1,6,\omega)$ and $F_0$: $\pi(3,6,\omega^3)$ & $\omega^3\unr(-1)\oplus\omega\unr(-1)$ & 7/2 & 3 \\
\hline
\end{tabular}

\subsubsection{Third case}
\label{32case3}

The third case can happen only when $p = 3$ and $k$ is odd and $3^2$ divides $k-5$.  In
this case the results of \cite{BG15} show that when  $v(a_3^2 -
(k-4)^2\frac{k-3}{2}3^3) = 3$ the reduction is equal to $I_{1,0}$.

Our guess
is that that we have two rational numbers $i(k) < j(k)$ such that:
\begin{itemize}
\item
if $1 < v(a_3) < 2$ and $v(a_3^2-(k-4)^2\frac{k-3}{2}3^3) < i(k)$ then
$\bar{V}_{k,a_3}$ is isomorphic to $I_{1,0}$.
\item
if $1 < v(a_3) < 2$ and $v(a_3^2-(k-4)^2\frac{k-3}{2}3^3) = i(k)$ then 
$\bar{V}_{k,a_3}$ is isomorphic to 
$I_{1,c}$
for some $c\in E^\times$.
\item
if $1 < v(a_3) < 2$ and $i(k) < v(a_3^2-(k-4)^2\frac{k-3}{2}3^3) < j(k)$ then
$\bar{V}_{k,a_3}^{ss}$ is isomorphic to
$\ind\omega_2^{2} $
\item
if $1 < v(a_3) < 2$ and $v(a_3^2-(k-4)^2\frac{k-3}{2}3^3) = j(k)$ then 
$\bar{V}_{k,a_3}$ is isomorphic to $I_{0,c}$
for some $c\in E^\times$.
\item
if $1 < v(a_3) < 2$ and $v(a_3^2-(k-4)^2\frac{k-3}{2}3^3) > j(k)$ then
$\bar{V}_{k,a_3}^{ss}$ is isomorphic to
$I_{0,0}$.
\end{itemize}

\noindent
\begin{tabular}{|c|c|c|c|c|c|c|}
\hline
$p$ & $k$ & $a_p$ & non-zero $F_i$s &  $\bar{V}_{k,a_p}^{ss}$ & $\delta$ & $n$ \\
\hline
$3$ & $23$ & $3\sqrt{30}\cdot 19$ & $F_2$: $I(\sigma_1)/(T^2+1)$ & $I_{0,0}$ & 7/2 & 4 \\
\hline
$3$ & $23$ & $3\sqrt{30}\cdot 19+3^4$ & $F_2$: $I(\sigma_1)/(T^2-2T+1)$ & $I_{0,2}$ & 9/2 & 6 \\
\hline
$3$ & $23$ & $3\sqrt{30}\cdot 19+3^3$ & $F_2$: $\pi(1,0,1)$ & $\ind\omega_2^2$ & 7/2 & 4 \\
\hline
$3$ & $23$ & $3\sqrt{30}\cdot 19+3^2$ & $F_3$: $I(\sigma_1(1))/(T^2-2T+1)$ & $I_{1,2}$ & 7/2 & 3 \\
\hline
\end{tabular}

\subsection{Examples in slope $5/2$}

We can expect that this case behaves like the cases of slope $1/2$ and
$3/2$, that is: for some values of $k$ modulo $p-1$, the reduction
depends on the valuation of $a_p^2-f(k)p^5$ for some integer $f(k)$.

We can find examples behaving in a way similar to \ref{32case1}:

\noindent
\begin{tabular}{|c|c|c|c|c|c|c|}
\hline
$p$ & $k$ & $a_p$ & non-zero $F_i$s &  $\bar{V}_{k,a_p}^{ss}$ & $\delta$ & $n$ \\
\hline
$7$ & $31$ & $7^2\sqrt{7}.44$ & $F_4$: $I(\sigma_5(3))/(T^2+1)$ & $I_{3,0}$ & 7/2 & 3 \\
\hline
$7$ & $31$ & $7^2\sqrt{7}.44+7^3$ & $F_4$: $I(\sigma_5(3))/(T^2-4T+1)$ & $I_{3,4}$ & 7/2 & 3 \\
\hline
$7$ & $31$ & $7^2\sqrt{7}$ & $F_4$: $\pi(5,0,\omega^3)$ & $\ind\omega_2^{30}$& 3 & 2 \\
\hline
$7$ & $31$ & $7^2\sqrt[3]{7}$ & $F_4$: $\pi(5,0,\omega^3)$ & $\ind\omega_2^{30}$& 8/3 & 2 \\
\hline
\end{tabular}

\medskip
We can also find examples behaving in a way similar to \ref{32case2}:

\noindent
\begin{tabular}{|c|c|c|c|c|c|c|}
\hline
$p$ & $k$ & $a_p$ & non-zero $F_i$s &  $\bar{V}_{k,a_p}^{ss}$ & $\delta$ & $n$ \\
\hline
$5$ & $27$ & $5^2\sqrt{5}.51$ & $F_4$: $I(\sigma_3(3))/(T^2+1)$ & $I_{3,0}$ & 9/2 & 3 \\
\hline
$5$ & $27$ & $5^2\sqrt{5}.51+5^4$ & $F_4$: $I(\sigma_3(3))/(T^2-T+1)$ & $I_{3,1}$ & 9/2 & 3 \\
\hline
$5$ & $27$ & $5^2\sqrt{5}.51+5^3\sqrt{5}$ & $F_4$: $\pi(3,0,\omega^3)$ & $\ind\omega_2^{22}$ & 9/2 & 3 \\
\hline
$5$ & $27$ & $5^2\sqrt{5}.51+5^3$ & $F_5$: $\pi(1,1,\omega^2)$ and $F_2$: $\pi(1,1,1)$ & $\omega^2\oplus 1$ & 9/2 & 3 \\
\hline
$5$ & $27$ & $5^2\sqrt{5}.51+2\cdot 5^3$ & $F_5$: $\pi(1,2,\omega^2)$ and $F_2$: $\pi(1,3,1)$ & $\omega^2\unr(3)\oplus\unr(2)$ & 9/2 & 3 \\
\hline
$5$ & $27$ & $5^2\sqrt{10}$ & $F_2$: $\pi(1,0,1)$ & $\ind\omega_2^{8}$& 4 & 3 \\
\hline
\end{tabular}

\smallskip

An interesting result of \cite{BG15} is also that for some values of
$k$, we have a reducible reduction (of the form $I_{1,0}$) for all $a_p$
with $1 < v(a_p) < 2$. This phenomenon seems to happen also for $2 <
v(a_p) < 3$. For example when $p=3$ and $k=17$ all the examples we have
computed with $2 < v(a_3) < 3$ give a reduction isomorphic to $I_{1,0}$.

\end{document}